\documentclass[12pt]{article}
\usepackage[a4paper,margin=27mm]{geometry}
\usepackage[T1]{fontenc}
\usepackage{amsmath,amssymb,amsthm,amscd,amsfonts}
\usepackage{times,microtype}
\usepackage{needspace}
\usepackage[hyperfootnotes=false,hidelinks]{hyperref}
\hypersetup{pdftitle={Dirichlet eigenvalue asymptotics on conformally compact Einstein manifolds},pdfauthor={Xiaoshang Jin}}
\numberwithin{equation}{section}
\newtheorem{theorem}{Theorem}[section]
\newtheorem{proposition}[theorem]{Proposition}
\newtheorem{lemma}[theorem]{Lemma}

\theoremstyle{remark}
\newtheorem{remark}[theorem]{Remark}
\newcommand{\dd}{\,d}
\newcommand{\dv}{\,dv}
\newcommand{\Ric}{\operatorname{Ric}}
\newcommand{\tr}{\operatorname{tr}}

\newcommand{\sgn}{\operatorname{sgn}}

\newcommand{\E}{Q}

\newcommand{\radialaverage}{\frac{\pi^2(1-e^{-2\ell})}{2\ell(\pi^2+\ell^2)}}
\title{Dirichlet eigenvalue asymptotics on conformally compact Einstein manifolds}
\author{Xiaoshang Jin\thanks{The author was supported by the Fundamental Research Funds for the Central Universities, HUST (No.\ 2025BRSXB002), and the National Natural Science Foundation of China (Grant No.\ 12471054).}}
\date{}
\begin{document}
\maketitle
\begin{abstract}
We study the first Dirichlet eigenvalues of large geodesic balls and fixed-width boundary layers on conformally compact Einstein manifolds of dimension $n+1.$ For large balls, we determine the exponential rate of convergence when the spectral bottom is less than $\frac{n^2}{4},$ i.e. the bottom of the continuous spectrum. When it equals the bottom of the continuous spectrum, the second coefficient has two possible values, according to whether the continued resolvent is regular or has a pole there. We also give examples with negative Yamabe conformal infinity realizing both threshold alternatives and construct metrics with a Ricci lower bound that have other second coefficients.

Under nonnegative Yamabe conformal infinity, we also find the exact second coefficient for the nonlinear $p$-Laplacian on large balls.

For layers defined by a geodesic boundary defining function, we compute the first correction and show that its sign characterizes nonnegative Yamabe conformal infinity.
\end{abstract}
\section{Introduction}
Let $\overline X^{n+1}$ be a compact smooth manifold of dimension $n+1$ with non-empty boundary $M$, and let $X=X^{n+1}$ be its interior. A complete Einstein metric $g^+$ on $X$, normalized by $\Ric(g^+)=-ng^+$, is called $C^{m,\alpha}$ conformally compact Einstein (CCE) if there is a smooth defining function $\rho$ satisfying
\[
 \rho>0\ \text{in }X,\qquad \rho=0\ \text{on }M,\qquad d\rho\ne0\ \text{on }M,
\]
such that $g=\rho^2g^+\in C^{m,\alpha}(\overline X)$ is a nondegenerate metric, where $m\ge0$ is an integer and $\alpha\in[0,1].$ The conformal class $[h]$ of $h=g|_{TM}$ is independent of $\rho$ and is called the conformal infinity. If the compactification is $C^2$, the Einstein equation gives $|d\rho|^2_g=1$ on $M$. The sectional curvatures of $g^+$ then tend to $-1$, so $g^+$ is asymptotically hyperbolic (AH). Classical accounts and constructions include Fefferman--Graham~\cite{FG}, Graham--Lee~\cite{GL}, Graham~\cite{Graham}.

CCE metrics provide a geometric setting for the AdS/CFT correspondence, which relates gravity in the interior to conformal field theory on the boundary; see Maldacena~\cite{Maldacena}. In this setting, Witten--Yau~\cite{WY} showed how positive scalar curvature at conformal infinity constrains the topology of the filling. Rigidity results identify conditions under which a filling with round spherical conformal infinity must be hyperbolic; see Qing~\cite{Qing}, Dutta--Javaheri~\cite{DJ}, Li--Qing--Shi~\cite{LQS} and Chen--Lai--Wang~\cite{CLW}. The existence problem asks which boundary conformal classes admit CCE fillings. Graham--Lee~\cite{GL} and Lee~\cite{LeeFredholm} established perturbative existence results, while Anderson~\cite{AndersonExistence} studied global existence in dimension four. Gursky--Han~\cite{GH} obtained obstructions to the existence of fillings with prescribed conformal infinity. Compactness and uniqueness under suitable geometric and topological assumptions were studied by Chang--Ge~\cite{CG}, Chang--Ge--Qing~\cite{CGQ} and Chang--Ge--Jin--Qing~\cite{CGJQ}. These works show how the conformal geometry of the boundary influences the geometry of the Einstein interior.

For a smooth AH metric $g^+,$ the spectral results of Mazzeo and Mazzeo--Melrose~\cite{Mazzeo91,MM} give
\begin{equation}\label{eq:spectrum}
 \sigma(-\Delta_{g^+})=\sigma_{\mathrm{pt}}(X)\cup\left[\frac{n^2}{4},\infty\right),
 \qquad \sigma_{\mathrm{pt}}(X)\subset\left(0,\frac{n^2}{4}\right),
\end{equation}
where $\sigma_{\mathrm{pt}}(X)$ is the finite set of $L^2$ eigenvalues, each of finite multiplicity. With the convention $\Delta=\operatorname{div}\nabla$, the spectral bottom is
\begin{equation}\label{eq:bottom}
 \lambda_1(X)=\inf\sigma(-\Delta_{g^+})
 =\inf\limits_{0\ne u\in C_c^\infty(X)}
 \frac{\int_X|du|_{g^+}^2\,dv_{g^+}}{\int_Xu^2\,dv_{g^+}}.
\end{equation}
For any increasing smooth relatively compact exhaustion $\Omega_j\nearrow X$, domain monotonicity and compactly supported test functions give
\begin{equation}\label{eq:exhaustion}
 \lambda_1(X)=\lim_{j\to\infty}\lambda_1(\Omega_j)
 =\lim_{R\to\infty}\lambda_1(B(o,R)).
\end{equation}
Here $\lambda_1(\Omega)$ is the first Dirichlet eigenvalue of $-\Delta_{g^+}$ on a relatively compact open set $\Omega\subset X$, defined variationally by
\[
 \lambda_1(\Omega)=\inf\limits_{0\ne u\in H_0^1(\Omega)}
 \frac{\int_\Omega|du|_{g^+}^2\,dv_{g^+}}{\int_\Omega u^2\,dv_{g^+}}.
\]
This definition also applies to nonsmooth balls and layers, with $H_0^1(\Omega)$ defined as the $H^1$ closure of $C_c^\infty(\Omega)$.
For $n\ge3$, the Yamabe constant of the conformal infinity is
\begin{equation}\label{eq:Yamabe-definition}
 Y(M,[h])=\inf_{0\ne u\in H^1(M)}
 \frac{\displaystyle\int_M\left(\frac{4(n-1)}{n-2}|du|_h^2+R_hu^2\right)\,dv_h}
 {\displaystyle\left(\int_M|u|^{\frac{2n}{n-2}}\,dv_h\right)^{\frac{n-2}{n}}},
\end{equation}
Here $R_h$ is the scalar curvature of $h$. The Yamabe constant depends only on $[h]$~\cite{LP}. In dimension $n=2$, nonnegative Yamabe infinity means that each boundary component admits a conformal representative of nonnegative scalar curvature.
Lee~\cite{Lee} proved that for a smoothly conformally compact Einstein manifold
$$
Y(M,[h])\ge0\quad\Rightarrow\quad\sigma_{\mathrm{pt}}(X)=\varnothing\quad\Leftrightarrow\quad\lambda_1(X)=\frac{n^2}{4}.
$$ This links a conformal inequality on the boundary to the spectral geometry of the Riemannian interior. Wang~\cite{WangLee} gave another proof of Lee's theorem.

For geodesic balls in hyperbolic space, Savo~\cite{Savo} obtained the two-term expansion below. The author extended it in \cite{Jin} to CCE manifolds with nonnegative Yamabe infinity:
\begin{equation}\label{eq:Jin}
 Y(M,[h])\ge0\quad\Longrightarrow\quad
 \lambda_1(B(o,R))=\frac{n^2}{4}+\frac{\pi^2}{R^2}+O(R^{-3})
 \qquad(R\longrightarrow\infty),
\end{equation}
for every fixed $o\in X$.

To distinguish the possible behavior at the continuous-spectrum threshold, consider the resolvent
\begin{equation}\label{eq:resolvent-definition}
 \mathbf{R}(s)=\bigl(-\Delta_{g^+}-s(n-s)\bigr)^{-1}.
\end{equation}
The spectral parameter $s$ is complex, with spectral value $s(n-s)$. For real $s>n$, this value is negative and \eqref{eq:resolvent-definition} is the bounded $L^2(X)$ inverse. More generally, on the physical half-plane $\operatorname{Re}s>\frac n2$ it is the $L^2$ resolvent wherever $s(n-s)\notin\sigma(-\Delta_{g^+})$. Its poles in this half-plane correspond to the discrete eigenvalues in \eqref{eq:spectrum}. Write
\[
 P=-\Delta_{g^+}-\frac{n^2}{4},\qquad
 s=\frac n2+r,\qquad
 \mathbf R\left(\frac n2+r\right)=(P+r^2)^{-1}.
\]
Mazzeo--Melrose~\cite{MM} continue this family meromorphically across $r=0$ on compactly supported smooth functions. The continued operators take values in $C^\infty_{\mathrm{loc}}(X)$ and suitable weighted spaces. We also refer to Graham--Zworski~\cite{GZ} for the corresponding scattering theory. This continuation is distinct from the $L^2$ inverse: $\frac{n^2}{4}$ belongs to the continuous spectrum, so $P^{-1}$ is not a bounded operator on $L^2(X)$. Near $r=0$, the continued family is analytic apart from possible isolated poles. Guillarmou--Qing~\cite[Lemma 3.2]{GQ} identify its behavior at $r=0$ with solutions of $Pu=0$ having the specified regular boundary decay. The boundary condition is made explicit below.

The large-ball behavior depends first on whether the spectral bottom is an isolated eigenvalue or the bottom of the continuous spectrum. In the latter case a further distinction is made by the continued resolvent.

\begin{theorem}\label{thm:balls}
Let $(X^{n+1},g^+)$ be connected and smoothly conformally compact Einstein, with $n\ge3$. Fix $o\in X$.
\begin{enumerate}
\item[\textup{(i)}] If $\lambda_1(X)<\frac{n^2}{4}$, set $\kappa=\sqrt{\frac{n^2}{4}-\lambda_1(X)}>0$, then there are constants $c_o,C_o>0$ such that, for all sufficiently large $R$,
\begin{equation}\label{eq:ball-exponential-estimate}
 c_o e^{-2\kappa R}\le\lambda_1(B(o,R))-\lambda_1(X)
 \le C_o e^{-2\kappa R}.
\end{equation}
\item[\textup{(ii)}] If $\lambda_1(X)=\frac{n^2}{4}$, then the continued resolvent $\mathbf R(s)$ is either analytic or has a simple pole at $s=\frac{n}{2}$, and
\begin{equation}\label{eq:ball-dichotomy}
 \begin{aligned}
 \mathbf{R}\text{ analytic at }\frac{n}{2}
 &\Longleftrightarrow\lambda_1(B(o,R))=
 \frac{n^2}{4}+\frac{\pi^2}{R^2}+O(R^{-3}),\\[6pt]
 \mathbf{R}\text{ has a pole at }\frac{n}{2}
 &\Longleftrightarrow\lambda_1(B(o,R))=
 \frac{n^2}{4}+\frac{\pi^2}{4R^2}+O(R^{-3}).
 \end{aligned}
\end{equation}
\end{enumerate}
\end{theorem}

The decay exponent in part~\textup{(i)} and the second coefficient in part~\textup{(ii)} are independent of the center. The constants $c_o,C_o$ and those in the remainders may depend on the fixed center $o$.

Choose a smooth geodesic defining function $x$ and put
\[
 P=-\Delta_{g^+}-\frac{n^2}{4},\qquad
 N_0=\{u\in x^{\frac{n}{2}}C^\infty(\overline X):Pu=0\}.
\]
Membership in $x^{\frac{n}{2}}C^\infty(\overline X)$ is independent of the smooth defining function. It is shown in \cite[Lemma 3.2]{GQ} that $\mathbf{R}(s)$ is analytic at $\frac{n}{2}$ exactly when $N_0=\{0\}$. A nonzero element of $N_0$ will be called a \emph{threshold resonant state}. This boundary condition distinguishes it from an arbitrary solution of $Pu=0$.

Guillarmou--Qing~\cite[Theorem 3.4]{GQ} also proved that $Y(M,[h])\ge0$ implies regularity at $\frac{n}{2}$. Together with Lee's theorem, Theorem~\ref{thm:balls}\textup{(ii)} therefore recovers \eqref{eq:Jin}. Their Theorem 1.1 characterizes $Y>0$ by the largest real scattering pole being below $\frac{n}{2}-1$. This condition is stronger than regularity at $\frac{n}{2}$. The large-ball alternative is not a statement of \cite{GQ}. When $\lambda_1(X)<\frac{n^2}{4}$, threshold regularity does not replace the isolated-eigenvalue behavior in part~\textup{(i)}.

\begin{remark}\label{prop:counterexample}
\begin{enumerate}
\item As an application of Theorem~\ref{thm:balls}, for every integer $n\ge3$ there is a smooth conformally compact hyperbolic manifold with $Y(M,[h])<0$ and
\begin{equation}\label{eq:counterexample}
 \lambda_1(B(o,R))=\frac{n^2}{4}+\frac{\pi^2}{R^2}+O(R^{-3}).
\end{equation}
Thus the converse of \eqref{eq:Jin} fails. These examples have positive solutions of $Pu=0$ by \cite[Theorem 2.1]{Sullivan}, but no nonzero solution in $x^{\frac{n}{2}}C^\infty(\overline X)$. Section~\ref{sec:examples} proves the assertion by applying Theorem~\ref{thm:balls} and also gives an explicit negative-Yamabe example with a threshold resonant state, realizing the other coefficient $\frac{\pi^2}{4}$.

\item For a general complete $(n+1)$-manifold $(X,g)$ with $\Ric_g\ge-ng$ and $\lambda_1(X)=\frac{n^2}{4}$, comparison gives
\[
 0\le\lambda_1(B(o,R))-\frac{n^2}{4}\le\frac{\pi^2}{R^2}+O(R^{-3}).
\]
Hence $R\bigl(\lambda_1(B(o,R))-\frac{n^2}{4}\bigr)\to0$. The $R^{-2}$ coefficient, when it exists, need not have either value in Theorem~\ref{thm:balls}. In Section~\ref{sec:ricci-examples}, we construct complete metrics in every dimension $n+1\ge3$ with sectional curvature $\ge-1$ and spectral bottom $\frac{n^2}{4}$. Their second coefficients fill a continuous interval of other values. These examples are not Einstein.
\end{enumerate}
\end{remark}
~\\
The second part of the paper concerns the nonlinear $p$-Laplacian on large geodesic balls.

For $1<p<\infty$, the principal $p$-eigenvalue is defined by
\[
 \lambda_{1,p}(X)=\inf\limits_{0\ne u\in C_c^\infty(X)}
 \frac{\int_X|du|_{g^+}^p\,dv_{g^+}}{\int_X|u|^p\,dv_{g^+}}.
\]
For a relatively compact open set $\Omega$, the same quotient with $u\in C_c^\infty(\Omega)$ defines its first Dirichlet eigenvalue $\lambda_{1,p}(\Omega)$. Hijazi, Montiel and Raulot~\cite[Theorem 9]{HMR} extended Lee's conclusion to this eigenvalue: for smoothly conformally compact Einstein manifolds of dimension $n+1,$
\begin{equation}\label{eq:p-bottom}
 Y(M,[h])\ge0\quad\Longrightarrow\quad\lambda_{1,p}(X)=\left(\frac{n}{p}\right)^p,\qquad 1<p<\infty.
\end{equation}
Their proof uses the characterization of nonnegative Yamabe infinity by the Cheeger constant being $n$. In \cite[Proposition 2.5]{Jinp} the author used Lee's positive eigenfunction and a divergence estimate to provide a new proof of \eqref{eq:p-bottom}. P\'erez-Ayala and Tyrrell~\cite[Theorem 1.2]{PT} proved the upper bound $\lambda_{1,p}(X)\le\left(\frac{n}{p}\right)^p$ for general AH manifolds and studied related estimates for their submanifolds. The author also obtained two-sided bounds in \cite[Theorem 1.4]{Jinp} for the second term of the large-ball $p$-eigenvalue. These bounds have different coefficients when $p\ne2.$ As an improvement, the following theorem identifies the coefficient exactly.

\begin{theorem}\label{thm:p-balls}
Let $(X^{n+1},g^+)$ be a $C^3$ conformally compact Einstein manifold, where $n\ge2$, with $Y(M,[h])\ge0$. Then for every fixed $p>1$ and $o\in X$,
\begin{equation}\label{eq:p-ball-expansion}
 \lambda_{1,p}(B(o,R))
 =\left(\frac np\right)^p
 +\frac{2(p-1)}p\left(\frac np\right)^{p-2}
  \frac{\pi^2}{R^2}+O(R^{-3})
 \qquad(R\longrightarrow\infty).
\end{equation}
\end{theorem}

The coefficient in \eqref{eq:p-ball-expansion} lies between the bounds of \cite[Theorem 1.4]{Jinp}:
\[
 \min\{p-1,1\}\le\frac{2(p-1)}{p}\le\frac{p}{2},
 \qquad 1<p<\infty.
\]
Both inequalities become equalities precisely when $p=2$. The remainder here is $O(R^{-3})$ for every fixed $p>1$. For $1<p<2$, the upper-bound remainder in \cite[Theorem 1.4]{Jinp} was $O(R^{-(p+1)})$. The proof in Section~\ref{sec:p-balls} uses the one-dimensional comparison theorem of Sakurai~\cite[Theorem 1.6]{Sakurai}. We compute the precise large-interval asymptotics of its model eigenvalue and modify the hyperbolic test function in \cite{Jinp}. At $p=2$, \eqref{eq:p-ball-expansion} agrees with \eqref{eq:Jin}.
\\~\\
The third part of the paper gives an if-and-only-if characterization of nonnegative Yamabe infinity from an eigenvalue expansion. We use layers of fixed width near conformal infinity. Their exponentially small correction retains information that is absent from the first two terms of \eqref{eq:Jin}.

The layer theorem uses a geodesic defining function and the corresponding boundary representative.

\Needspace{16\baselineskip}
\begin{theorem}\label{thm:main}
Let $(X^{n+1},g^+)$ be a CCE manifold with $n\ge3$. Suppose that $x^2g^+$ is a $C^3$ geodesic conformal compactification with boundary metric $h=(x^2g^+)|_{TM}$. Set $r=-\ln x$ and $A_{R,\ell}=\{q\in X:R<r(q)<R+\ell\}$. Then for each fixed $\ell>0,$
\begin{equation}\label{eq:main}
 \lambda_1(A_{R,\ell})
 =\frac{n^2}{4}+\frac{\pi^2}{\ell^2}+C_\ell e^{-2R}+O(e^{-3R}),\ \ R\to\infty
\end{equation}
where
\begin{equation}\label{eq:main-coefficient}
 C_\ell=\frac{\pi^2(1-e^{-2\ell})}{2\ell(\pi^2+\ell^2)}\,\mu_1(L_h).
\end{equation}
and $L_h=-\Delta_h+\frac{n-2}{4(n-1)}R_h$ is the conformal Laplacian with first eigenvalue $\mu_1(L_h)$. Moreover,
\begin{equation}\label{eq:signintro}
 Y(M,[h])\ge0\quad\Longleftrightarrow\quad C_\ell\ge0.
\end{equation}
\end{theorem}

Although $C_\ell$ depends on the representative $h$, its sign is conformally invariant. Theorem~\ref{thm:main} requires neither $\lambda_1(X)=\frac{n^2}{4}$ nor a nonresonance assumption. The remainder is uniform when $\ell$ ranges over a compact subset of $(0,\infty)$. Section~\ref{sec:proof} proves the expansion by constructing a positive test function from the first boundary eigenfunction and the perturbed first radial mode. Its Rayleigh quotient gives the upper bound, and Barta's inequality gives the lower bound. The relation with ordinary distance layers is discussed briefly at the end of that section.
\section{Large-ball eigenvalues and examples}\label{sec:counterexamples}
A defining function $x$ is \emph{geodesic} if $|dx|_{x^2g^+}=1$ throughout a boundary collar. If $x^2g^+$ is a $C^{m,\alpha}$ conformal compactification, we call it a $C^{m,\alpha}$ \emph{geodesic conformal compactification}. It is well known that a $C^{m,\alpha}$ conformally compact Einstein metric, with $m\ge2$ and $0\le\alpha<1$, admits a $C^{m-1,\alpha}$ geodesic conformal compactification near the boundary inducing the same boundary metric; see Graham--Lee~\cite[Lemma 5.2]{GL} and Anderson~\cite[Section 1, (1.9)--(1.11)]{Anderson}. In geodesic collar coordinates, $g^+=x^{-2}(dx^2+h_x)$. In this section the compactification is smooth, and we choose a smooth geodesic defining function $x$.
\subsection{An isolated spectral bottom}
\begin{proof}[Proof of Theorem~\ref{thm:balls}\textup{(i)}]
Write $\kappa=\sqrt{\frac{n^2}{4}-\lambda_1(X)}>0$. The isolated bottom eigenvalue is simple, and its eigenfunction $\phi$ can be chosen positive with $\|\phi\|_{L^2(X)}=1$. The boundary expansion for an $L^2$ eigenfunction~\cite[Section 3, (3.14)]{GZ} gives
\[
 \phi=x^{\frac{n}{2}+\kappa}F,\qquad
 F=b(y)+O(x)\in C^\infty(\overline X),\qquad b\not\equiv0,
\]
where $y$ denotes local coordinates on $M$ and $b=F|_M$. To justify the last assertion, if $b$ vanished identically, the indicial equation would force every Taylor coefficient of $F$ to vanish: at the $j$th step the nonzero indicial factor is $-j(2\kappa+j)$ for $j\ge1$. Unique continuation at infinity~\cite{Mazzeo91} would then give $\phi=0$, contradicting its normalization.

In the geodesic coordinate $t=-\ln x$, the metric is $dt^2+e^{2t}h_{e^{-t}}$. Since $h_x=h+O(x^2)$,
\[
 dv_{g^+}=e^{nt}(1+O(e^{-2t}))\,dt\,dv_h,
 \qquad \phi=e^{-(\frac{n}{2}+\kappa)t}(b(y)+O(e^{-t})).
\]
The errors are uniform on the compact boundary. Hence
\begin{align}\label{eq:ground-tail-coordinate}
 \int_{\{t>R\}}\phi^2\,dv_{g^+}
 &=\int_R^\infty\!\int_M e^{-2\kappa t}
 (b(y)+O(e^{-t}))^2(1+O(e^{-2t}))\,dv_h\,dt\notag\\
 &=\left(\int_Mb^2\,dv_h\right)\int_R^\infty e^{-2\kappa t}\,dt
 +O\!\left(\int_R^\infty e^{-(2\kappa+1)t}\,dt\right)\notag\\
 &=\frac{e^{-2\kappa R}}{2\kappa}\int_Mb^2\,dv_h
 +O(e^{-(2\kappa+1)R})\notag\\
 &=\frac{e^{-2\kappa R}}{2\kappa}
 \left(\int_M b^2\,dv_h+O(e^{-R})\right).
\end{align}
Extend $t$ smoothly over the compact interior. For a fixed $o\in X$, one has $|d_{g^+}(o,\cdot)-t(\cdot)|\le C$ on $X$ for some constant $C>0$. Thus, for large $R$,
\[
 \{t>R+C\}\subset X\setminus B(o,R)\subset\{t>R-C-1\}.
\]
Since $\int_Mb^2\,dv_h>0$, \eqref{eq:ground-tail-coordinate} gives positive constants, depending on $o$, such that
\begin{equation}\label{eq:ground-tail-ball}
 c_o e^{-2\kappa R}\le
 \int_{X\setminus B(o,R)}\phi^2\,dv_{g^+}
 \le C_o e^{-2\kappa R}.
\end{equation}

Choose a Lipschitz cutoff $\chi$, equal to one on $B(o,R-2)$, supported in $B(o,R-1)$, and satisfying $0\le\chi\le1$ and $|d\chi|\le2$ almost everywhere. Integration by parts gives
\[
 \int_X|d(\chi\phi)|^2\,dv_{g^+}
 -\lambda_1(X)\int_X\chi^2\phi^2\,dv_{g^+}
 =\int_X\phi^2|d\chi|^2\,dv_{g^+}.
\]
Since $\chi\phi\in H_0^1(B(o,R))$, its Rayleigh quotient yields
\begin{align*}
 \lambda_1(B(o,R))-\lambda_1(X)
 &\le\frac{\int_X|d(\chi\phi)|^2\,dv_{g^+}}
 {\int_X\chi^2\phi^2\,dv_{g^+}}-\lambda_1(X)\\
 &=\frac{\int_X\phi^2|d\chi|^2\,dv_{g^+}}
 {\int_X\chi^2\phi^2\,dv_{g^+}}\\
 &\le\frac{4\int_{X\setminus B(o,R-2)}\phi^2\,dv_{g^+}}
 {1-\int_{X\setminus B(o,R-2)}\phi^2\,dv_{g^+}}\\
 &\le C_o e^{-2\kappa R}.
\end{align*}
In the last step, \eqref{eq:ground-tail-ball} bounds the numerator by $4C_o e^{4\kappa}e^{-2\kappa R}$ and the denominator from below by $\frac{1}{2}$ for all sufficiently large $R$; fixed factors are absorbed into $C_o$.

For the reverse inequality, isolation and simplicity imply a positive spectral gap
\[
 \eta:=\inf\bigl(\sigma(-\Delta_{g^+})\setminus\{\lambda_1(X)\}\bigr)
 -\lambda_1(X)>0.
\]
By the spectral theorem, the quadratic form on the orthogonal complement of $\phi$ satisfies
\[
 \int_X|dw|^2\,dv_{g^+}\ge(\lambda_1(X)+\eta)\|w\|_{L^2}^2,
 \qquad \forall w\in H^1(X),\quad \langle w,\phi\rangle_{L^2}=0.
\]
Take any real $v\in H_0^1(B(o,R))$ with $\|v\|_{L^2}=1$ and extend it by zero to $X$. The definition of $H_0^1$ by closure ensures that this extension belongs to $H^1(X)$. Decompose
\[
 v=a\phi+w,\qquad a=\langle v,\phi\rangle_{L^2},\qquad
 w\perp\phi,\qquad \|w\|_{L^2}^2=1-|a|^2.
\]
The divergence formula, first for compactly supported smooth functions and then by density in $H^1(X)$, gives
\[
 \int_X\langle d\phi,dw\rangle\,dv_{g^+}
 =\lambda_1(X)\langle\phi,w\rangle_{L^2}=0.
\]
Consequently,
\begin{align*}
 \int_X|dv|^2\,dv_{g^+}-\lambda_1(X)
 &=\lambda_1(X)|a|^2+\int_X|dw|^2\,dv_{g^+}-\lambda_1(X)\\
 &=\int_X|dw|^2\,dv_{g^+}-\lambda_1(X)\|w\|_{L^2}^2\\
 &\ge\eta\|w\|_{L^2}^2
 =\eta(1-|a|^2).
\end{align*}
Since $v=0$ almost everywhere outside $B(o,R)$, Cauchy--Schwarz gives
\[
 |a|^2=\left|\int_{B(o,R)}v\phi\,dv_{g^+}\right|^2
 \le\left(\int_{B(o,R)}v^2\,dv_{g^+}\right)
 \left(\int_{B(o,R)}\phi^2\,dv_{g^+}\right)
 =\int_{B(o,R)}\phi^2\,dv_{g^+}.
\]
Using $\|\phi\|_{L^2}=1$, we obtain
\begin{align*}
 \int_X|dv|^2\,dv_{g^+}-\lambda_1(X)
 &\ge\eta(1-|a|^2)\\
 &\ge\eta\left(1-\int_{B(o,R)}\phi^2\,dv_{g^+}\right)\\
 &=\eta\int_{X\setminus B(o,R)}\phi^2\,dv_{g^+}\\
 &\ge\eta c_o e^{-2\kappa R}.
\end{align*}
Taking the infimum over all such $v$ yields
\[
 \lambda_1(B(o,R))-\lambda_1(X)
 =\inf_{\substack{v\in H_0^1(B(o,R))\\\|v\|_{L^2}=1}}
 \left(\int_X|dv|^2\,dv_{g^+}-\lambda_1(X)\right)
 \ge\eta c_o e^{-2\kappa R}.
\]
Absorbing $\eta$ into $c_o$ and combining this with the upper bound proves
\begin{equation}\label{eq:ball-exponential-bounds}
 c_o e^{-2\kappa R}\le
 \lambda_1(B(o,R))-\lambda_1(X)
 \le C_o e^{-2\kappa R}.
\end{equation}
This proves \eqref{eq:ball-exponential-estimate}. No smoothness of the distance function or of the ball boundary is needed: the upper bound uses a Lipschitz cutoff and the lower bound uses only extension by zero from $H_0^1(B(o,R))$.
\end{proof}

\subsection{The threshold resolvent and a positive function}
We now assume $\lambda_1(X)=\frac{n^2}{4}$ and prove part~\textup{(ii)}. For real $r>0$, the resolvent in \eqref{eq:resolvent-definition} satisfies
\[
 \mathbf{R}\left(\frac{n}{2}+r\right)=(P+r^2)^{-1},
 \qquad P=-\Delta_{g^+}-\frac{n^2}{4}\ge0.
\]
By \cite[Lemma 3.2 and its proof]{GQ},
\[
 \mathbf{R}(s)\text{ analytic at }\frac{n}{2}
 \quad\Longleftrightarrow\quad N_0=\{0\}.
\]
If a pole occurs at $s=\frac{n}{2}$, it is simple. Its residue is a nonzero finite-rank operator with range in $N_0$.

For real $r>0$ and $f\ge0$, let $u=(P+r^2)^{-1}f$ and
$u_-=\max\{-u,0\}\in H^1(X)$. Testing $(P+r^2)u=f$ with $u_-$
and integrating by parts gives
\[
 0\le\int_Xfu_-\,dv_{g^+}
 =-\left[\int_X\left(|du_-|_{g^+}^2
          -\frac{n^2}{4}u_-^2\right)dv_{g^+}
       +r^2\int_Xu_-^2\,dv_{g^+}\right]\le0.
\]
The second inequality uses $P\ge0$. Hence $u_-=0$ and
$(P+r^2)^{-1}f=u\ge0$.

In the regular case, choose a fixed smooth compactly supported function $0\ne f\in C_c^\infty(X)$ with $f\ge0$, independent of $r$, and define
\[
 \phi=\mathbf{R}\left(\frac{n}{2}\right)f
 =\lim_{r\downarrow0}(P+r^2)^{-1}f\ge0.
\]
The continuation and elliptic regularity give convergence in $C^\infty$ on compact subsets. Passing to the limit in $(P+r^2)(P+r^2)^{-1}f=f$ gives $P\phi=f$, so $\phi\not\equiv0$.

In the singular case, let $A=\operatorname*{Res}_{s=\frac{n}{2}}\mathbf{R}(s)\ne0$. The Laurent expansion is
\[
 \mathbf{R}\left(\frac{n}{2}+r\right)=\frac{A}{r}+\text{an analytic family}.
\]
Choose a fixed smooth compactly supported function $0\ne f\in C_c^\infty(X)$ with $f\ge0$ and $Af\ne0$, independent of $r$. Such a choice exists: every real test function is the difference of two nonnegative smooth compactly supported functions, so vanishing on all of them would imply $A=0$. Define
\[
 \phi=Af=\lim_{r\downarrow0}r(P+r^2)^{-1}f\ge0.
\]
Multiplying the resolvent equation by $r$ gives
\[
 P\bigl(r(P+r^2)^{-1}f\bigr)
 =rf-r^2\bigl(r(P+r^2)^{-1}f\bigr)\longrightarrow0.
\]
Hence $P\phi=0$ and $\phi\not\equiv0$.
\[
 \text{In either construction, $f$ is fixed once and for all.}
\]
Neither $f$ nor its support depends on the spectral parameter $r$ or on the truncation radius $R$ used below. In both cases,
\[
 -\Delta_{g^+}\phi=\frac{n^2}{4}\phi+P\phi\ge0.
\]
The strong minimum principle on the connected manifold $X$ therefore gives $\phi>0$. The boundary mapping properties of the continued resolvent, or of its residue, give the following expansion; see \cite[Lemma 3.2 and its proof]{GQ}:
\begin{equation}\label{eq:positive-threshold}
 \phi=x^{\frac{n}{2}}F,\quad F\in C^\infty(\overline X),\qquad
 P\phi=\begin{cases}f,&N_0=\{0\},\\0,&N_0\ne\{0\}.
 \end{cases}
\end{equation}
Since $f$ is compactly supported, $P\phi=0$ in a collar of infinity. The following lemma records the boundary information needed in the variational argument.

\begin{lemma}\label{lem:positive-boundary-coefficient}
Let $(X^{n+1},g^+)$ be a connected smoothly conformally compact Einstein manifold, and let $x$ be a smooth geodesic defining function, so that
\[
 g^+=x^{-2}(dx^2+h_x)
\]
near $M=\partial\overline X$. Put $P=-\Delta_{g^+}-\frac{n^2}{4}$. Suppose that $\phi>0$ in $X$,
\[
 \phi=x^{\frac{n}{2}}F,\qquad F\in C^\infty(\overline X),
\]
and $P\phi=0$ in a collar of $M$. Then, with $b=F|_M$,
\begin{equation}\label{eq:positive-b}
 \phi=x^{\frac{n}{2}}\bigl(b(y)+O(x^2)\bigr),
 \qquad b>0\quad\text{on }M.
\end{equation}
\end{lemma}

\begin{proof}
Set $H=\partial_x\ln\sqrt{\frac{\det h_x}{\det h}}$. Since $h_x=h+O(x^2)$ and $H=O(x)$, substitution of $\phi=x^{\frac{n}{2}}F$ into $P\phi=0$ gives
\begin{equation}\label{eq:lift-F}
 F_{xx}+\left(\frac{1}{x}+H\right)F_x
 +\Delta_{h_x}F+\frac{nH}{2x}F=0.
\end{equation}
We multiply \eqref{eq:lift-F} by $x$ and let $x\downarrow0$. This gives $F_x(0,y)=0$. Taylor's formula then gives
\[
 F(x,y)=b(y)+O(x^2).
\]
It remains to prove that $b>0$ on $M$.
The equation \eqref{eq:lift-F} for $F$ contains the singular term $\frac{1}{x}F_x$. To remove this singularity, we introduce $z\in\mathbb R^2$ and define
\[
 \widetilde F(z,y)=F(|z|,y).
\]
For $z\ne0$, the radial Laplacian in two dimensions satisfies
\[
 \Delta_z\widetilde F
 =F_{xx}(|z|,y)+\frac{1}{|z|}F_x(|z|,y).
\]
Thus \eqref{eq:lift-F} becomes
\[
 \Delta_z\widetilde F+\Delta_{h_{|z|}}\widetilde F
 +\frac{H(|z|,y)}{|z|}z\cdot\nabla_z\widetilde F
 +\frac{nH(|z|,y)}{2|z|}\widetilde F=0.
\]
The identity $F_x(0,y)=0$ ensures that $\widetilde F$ is $C^2$ across $z=0$. Since $H=O(x)$ and $H$ is smooth, the coefficients also extend continuously across $z=0$. The equation therefore holds there by continuity. Its operator is locally uniformly elliptic and has bounded lower-order coefficients.

Since $\widetilde F\ge0$ and its equation has a bounded zeroth-order coefficient, we may lower that coefficient to a negative constant so that $\widetilde F$ is a supersolution. The strong minimum principle~\cite[Theorem 3.5]{GT}, together with $\widetilde F(z,y)>0$ for $z\ne0$, gives $b(y)=\widetilde F(0,y)>0$ for every $y\in M$. This proves \eqref{eq:positive-b}.
\end{proof}

\subsection{Uniform variational comparison}
Let $x$ be a geodesic defining function and set $t=-\ln x$. Choose $t_0>0$ such that $X\setminus K$ is identified with $(t_0,\infty)\times M$, where
\[
 K=X\setminus\bigl((t_0,\infty)\times M\bigr)
\]
is compact. We increase $t_0$ so that $K$ is a connected smooth compact region and $\operatorname{supp}f\subset\operatorname{int}K$. For $R>t_0$, define
\[
 X_R:=X\setminus\bigl([R,\infty)\times M\bigr).
 \qquad\text{Then}\qquad \partial X_R=\{R\}\times M.
\]

All product sets here are understood through the collar identification. In these coordinates,
\[
 g^+=dt^2+e^{2t}h_{e^{-t}},\qquad
 d\mu=\phi^2\,dv_{g^+},\qquad W=\frac{P\phi}{\phi}.
\]
Thus $W=0$ in the pole case; in the analytic case $W=\frac{f}{\phi}\ge0$ is supported in the interior of $K$ and is not identically zero. The function $\phi$, the region $K$, and all constants in the compact-region estimates below are independent of $R$. On the collar, \eqref{eq:positive-b} gives
\begin{equation}\label{eq:weighted-density}
 d\mu=\sigma(t,y)\,dt\,dv_h,\qquad
 \sigma(t,y)=b(y)^2(1+O(e^{-2t})),\quad \min_M b>0.
\end{equation}
For fixed $R$, $\phi$ is smooth and strictly positive on $\overline{X_R}$. Multiplication by $\phi$ is therefore an isomorphism of $H_0^1(X_R)$ onto itself. For real $v\in C_c^\infty(X_R)$, set $u=\phi v$. Expanding the square and using
$d(\phi v^2)=v^2d\phi+2\phi v\,dv$, we obtain
\begin{align}\label{eq:ground-balls}
 &\int_{X_R}\left(|du|^2-\frac{n^2}{4}u^2\right)dv_{g^+}\notag\\
 &=\int_{X_R}\left(\phi^2|dv|^2+2\phi v\langle d\phi,dv\rangle
       +v^2|d\phi|^2-\frac{n^2}{4}\phi^2v^2\right)dv_{g^+}\notag\\
 &=\int_{X_R}\phi^2|dv|^2\,dv_{g^+}
   +\int_{X_R}\left(\langle d\phi,d(\phi v^2)\rangle
       -\frac{n^2}{4}\phi^2v^2\right)dv_{g^+}\notag\\
 &=\int_{X_R}\phi^2|dv|^2\,dv_{g^+}
   +\int_{X_R}\phi v^2\left(-\Delta_{g^+}\phi-\frac{n^2}{4}\phi\right)dv_{g^+}\notag\\
 &=\int_{X_R}|dv|^2\,d\mu+\int_KWv^2\,d\mu \notag\\
 :&=q(v).
\end{align}
There is no boundary term because $v$ has compact support. Density of $C_c^\infty(X_R)$ in $H_0^1(X_R)$ extends the identity to all real $v\in H_0^1(X_R)$: on this fixed compact region the coefficients $\phi$, $d\phi$, and $W$ are bounded, and $\phi$ is bounded away from zero. Since $\int_{X_R}u^2\,dv_{g^+}=\int_{X_R}v^2\,d\mu$, the Dirichlet variational principle gives
\[
 \lambda_1(X_R)-\frac{n^2}{4}
 =\inf_{\substack{v\in H_0^1(X_R)\\v\ne0}}
 \frac{q(v)}{\displaystyle\int_{X_R}v^2\,d\mu}.
\]
In particular, $q(v)\ge0$. We now estimate this quotient with constants uniform as $R\to\infty$.

On the fixed compact region, the trace and Poincar\'e inequalities yield
\begin{equation}\label{eq:compact-res}
 \int_Kv^2\,d\mu\le C\left(\int_K|dv|^2\,d\mu+
 \int_M b^2|v(t_0,y)|^2\,dv_h\right).
\end{equation}
To see this explicitly, let $\overline v=\frac{1}{\mu(K)}\int_Kv\,d\mu$. Since $\phi$ is smooth and strictly positive on the fixed compact region $K$, the measures $d\mu=\phi^2dv_{g^+}$ and $dv_{g^+}$ are uniformly equivalent on $K$. Likewise, the positive measure $b^2dv_h$ is uniformly equivalent to the induced measure on $\partial K=\{t_0\}\times M$. The Poincar\'e inequality on the connected region $K$, followed by the trace theorem, gives
\[
 \int_K|v-\overline v|^2\,d\mu
 +\int_Mb^2|v(t_0,y)-\overline v|^2\,dv_h
 \le C\int_K|dv|^2\,d\mu.
\]
Moreover,
\[
 |\overline v|^2\int_Mb^2\,dv_h
 \le2\int_Mb^2|v(t_0,y)|^2\,dv_h
    +2\int_Mb^2|v(t_0,y)-\overline v|^2\,dv_h.
\]
Combining these inequalities with
$\int_Kv^2\,d\mu=\int_K|v-\overline v|^2\,d\mu+\mu(K)|\overline v|^2$
proves \eqref{eq:compact-res}.

In the analytic case there is the stronger estimate
\begin{equation}\label{eq:compact-reg}
 \int_Kv^2\,d\mu+\int_Mb^2|v(t_0,y)|^2\,dv_h
 \le C\int_K(|dv|^2+Wv^2)\,d\mu.
\end{equation}
Here $W\in L^\infty(K)$, $W\ge0$, and $\int_KW\,d\mu>0$. The classical Poincar\'e inequality with a nonnegative potential follows from
\begin{align*}
 |\overline v|^2\int_KW\,d\mu
 &\le2\int_KWv^2\,d\mu
     +2\int_KW|v-\overline v|^2\,d\mu\\
 &\le2\int_KWv^2\,d\mu+C\int_K|dv|^2\,d\mu
\end{align*}
and the ordinary Poincar\'e inequality. Indeed,
\begin{align*}
 \int_Kv^2\,d\mu
 &=\int_K|v-\overline v|^2\,d\mu+\mu(K)|\overline v|^2\\
 &\le C\int_K\bigl(|dv|^2+Wv^2\bigr)\,d\mu.
\end{align*}
The trace theorem and the equivalence of the weighted and unweighted measures then give
\begin{align*}
 \int_Mb^2|v(t_0,y)|^2\,dv_h
 &\le C\left(\int_Kv^2\,d\mu+\int_K|dv|^2\,d\mu\right)\\
 &\le C\int_K\bigl(|dv|^2+Wv^2\bigr)\,d\mu,
\end{align*}
which proves \eqref{eq:compact-reg}, including its boundary term.

For each $y$, change the radial coordinate by
\begin{equation}\label{eq:optical-coordinate}
 s(t,y)=\int_{t_0}^t\frac{b(y)^2}{\sigma(\tau,y)}\,d\tau,
 \qquad s(R,y)=R+O(1).
\end{equation}
Indeed, \eqref{eq:weighted-density} implies
\[
 \frac{ds}{dt}=\frac{b^2}{\sigma}=1+O(e^{-2t}),\qquad
 |s(t,y)-(t-t_0)|
 \le C\int_{t_0}^\infty e^{-2\tau}\,d\tau\le C.
\]
The change of variable is strictly increasing, its derivative and inverse derivative are uniformly bounded, and all estimates are uniform in $y$ and $R$. Write
\[
 v(t,y)=V(s(t,y),y).
\]
The chain rule gives
$v_t=V_s\frac{b^2}{\sigma}$ and $dt=\frac{\sigma}{b^2}\,ds$. Thus
\begin{equation}\label{eq:exact-radial}
 \int_{t_0}^R\sigma|v_t|^2\,dt
 =\int_0^{s(R,y)}\sigma\frac{b^4}{\sigma^2}|V_s|^2
       \frac{\sigma}{b^2}\,ds
 =b^2\int_0^{s(R,y)}|V_s|^2\,ds.
\end{equation}
Since $\bigl(\frac{\sigma(t(s,y),y)}{b^2}\bigr)^2=1+O(e^{-2s})$, a change of variables, the fundamental theorem of calculus, and Cauchy--Schwarz give
\begin{align}\label{eq:mass-error}
 &\left|\int_{X_R\setminus K}v^2d\mu-
 \int_Mb^2\int_0^{s(R,y)}V^2ds\,dv_h\right|\notag\\
 &=\left|\int_Mb^2\int_0^{s(R,y)}
 \left[\left(\frac{\sigma(t(s,y),y)}{b^2}\right)^2-1\right]
 V^2\,ds\,dv_h\right|\notag\\
 &\le C\int_Mb^2\int_0^{s(R,y)}e^{-2s}
 \left(|V(0,y)|^2+s\int_0^{s(R,y)}|V_s(\tau,y)|^2\,d\tau\right)
 ds\,dv_h\notag\\
 &\le C\int_Mb^2\left(|V(0,y)|^2+
 \int_0^{s(R,y)}|V_s(s,y)|^2\,ds\right)dv_h.
\end{align}
The $y$-dependence of $s$ causes no loss, since the nonnegative tangential energy may be discarded:
\begin{align}
 q(v)
 &=\int_K(|dv|^2+Wv^2)\,d\mu+\int_M\int_{t_0}^R
   \left(|v_t|^2+e^{-2t}|d_yv|_{h_{e^{-t}}}^2\right)\sigma\,dt\,dv_h\notag\\
 &\ge\int_K(|dv|^2+Wv^2)\,d\mu
   +\int_M\int_{t_0}^R\sigma|v_t|^2\,dt\,dv_h\notag\\
 &=\int_K(|dv|^2+Wv^2)\,d\mu
   +\int_Mb^2\int_0^{s(R,y)}|V_s|^2\,ds\,dv_h.
   \label{eq:radial-lower}
\end{align}
All these formulas extend from smooth functions to $H_0^1(X_R)$ by density. The outer trace condition becomes $V(s(R,y),y)=0$ for almost every $y$.

\begin{lemma}\label{lem:truncated-lower-bounds}
Assume the hypotheses of Theorem~\ref{thm:balls}\textup{(ii)}, and let $\phi$ be the positive function constructed in Subsection~2.2. There exist constants $C>0$ and $R_0>t_0$, independent of $R$, such that the following statements hold for every $R\ge R_0$:
\begin{itemize}
\item If $\mathbf{R}(s)$ has a simple pole at $s=\frac{n}{2}$, then
\begin{equation}\label{eq:lower-res}
 \lambda_1(X_R)-\frac{n^2}{4}
 \ge\frac{\pi^2}{4(R+C)^2}.
\end{equation}
\item If $\mathbf{R}(s)$ is analytic at $s=\frac{n}{2}$, then
\begin{equation}\label{eq:lower-reg}
 \lambda_1(X_R)-\frac{n^2}{4}
 \ge\frac{\pi^2}{(R+C)^2}.
\end{equation}
\end{itemize}
\end{lemma}

\begin{proof}
\begin{itemize}
\item Suppose that $\mathbf{R}(s)$ has a simple pole at $s=\frac{n}{2}$, equivalently $N_0\ne\{0\}$. Here $W=0$. By \eqref{eq:compact-res}, \eqref{eq:mass-error}, and \eqref{eq:radial-lower}, there exist constants $C_1,C_2>0$ such that
\[
 \int_{X_R}v^2\,d\mu
 \le\int_Mb^2\left(\int_0^{s(R,y)}V^2\,ds
 +C_1|V(0,y)|^2\right)dv_h+C_2q(v).
\]
For fixed $y$, extend $V$ evenly across $s=0$ on the symmetric interval
with endpoints $\pm s(R,y)$. Since $V(s(R,y),y)=0$, the even extension
belongs to $H_0^1(-s(R,y),s(R,y))$. The Wirtinger inequality and the
Cauchy--Schwarz inequality give
\begin{align*}
 \int_0^{s(R,y)}|V|^2\,ds
 &=\frac{1}{2}\int_{-s(R,y)}^{s(R,y)}|V|^2\,ds
 \le\frac{2s(R,y)^2}{\pi^2}
       \int_{-s(R,y)}^{s(R,y)}|V_s|^2\,ds\\
 &=\frac{4s(R,y)^2}{\pi^2}
       \int_0^{s(R,y)}|V_s|^2\,ds,\\
 |V(0)|^2
 &=\left|\int_0^{s(R,y)} V_s\,ds\right|^2
 \le s(R,y)\int_0^{s(R,y)}|V_s|^2\,ds.
\end{align*}
By \eqref{eq:optical-coordinate}, there is a constant $C_3>0$ such that
$s(R,y)\le R+C_3$ uniformly. Hence \eqref{eq:radial-lower} implies
\begin{align*}
 \int_{X_R}v^2\,d\mu
 &\le\left[
 \frac{4(R+C_3)^2}{\pi^2}
 +C_1(R+C_3)+C_2\right]q(v)\\
 &\le\frac{4(R+C_4)^2}{\pi^2}q(v)
\end{align*}
for a fixed $C_4>0$ and all sufficiently large $R$. Taking the infimum of the Rayleigh quotient gives
\[
 \lambda_1(X_R)-\frac{n^2}{4}
 \ge\frac{\pi^2}{4(R+C_4)^2},
\]
which proves \eqref{eq:lower-res} after enlarging the constant in its statement.

\item Suppose that $\mathbf{R}(s)$ is analytic at $s=\frac{n}{2}$, equivalently $N_0=\{0\}$. Let $C_5>0$ be a constant for which \eqref{eq:compact-reg} holds. Combining \eqref{eq:compact-reg} with \eqref{eq:radial-lower} gives
\[
 q(v)\ge\int_Mb^2\left(
 \int_0^{s(R,y)}|V_s|^2\,ds+
 \frac{1}{C_5}|V(0,y)|^2\right)dv_h.
\]
Moreover, \eqref{eq:compact-reg}, \eqref{eq:mass-error}, and \eqref{eq:radial-lower} give a constant $C_6>0$ such that
\[
 \int_{X_R}v^2\,d\mu
 \le\int_Mb^2\int_0^{s(R,y)}V^2\,ds\,dv_h+C_6q(v).
\]
For fixed $y$, define
\[
 \widehat V(s)=
 \begin{cases}
 V(0)\left(1+\dfrac{s}{C_5}\right),&-C_5\le s\le0,\\[4pt]
 V(s),&0\le s\le s(R,y).
 \end{cases}
\]
Then $\widehat V\in H_0^1(-C_5,s(R,y))$. Applying the Wirtinger
inequality on this interval and discarding the nonnegative mass on
$[-C_5,0]$ gives
\begin{align*}
 \int_0^{s(R,y)}|V|^2\,ds
 &\le\int_{-C_5}^{s(R,y)}|\widehat V|^2\,ds\\
 &\le\frac{(s(R,y)+C_5)^2}{\pi^2}
 \int_{-C_5}^{s(R,y)}|\widehat V_s|^2\,ds\\
 &=\frac{(s(R,y)+C_5)^2}{\pi^2}
 \left(\int_0^{s(R,y)}|V_s|^2\,ds+
 \frac{1}{C_5}|V(0)|^2\right).
\end{align*}
Using $s(R,y)\le R+C_3$, integrating over $M$, and applying the preceding two inequalities, we obtain
\begin{align*}
 \int_{X_R}v^2\,d\mu
 &\le\left[\frac{(R+C_3+C_5)^2}{\pi^2}+C_6\right]q(v)\\
 &\le\frac{(R+C_7)^2}{\pi^2}q(v)
\end{align*}
for a fixed $C_7>0$ and all sufficiently large $R$. Taking the infimum gives
\[
 \lambda_1(X_R)-\frac{n^2}{4}
 \ge\frac{\pi^2}{(R+C_7)^2},
\]
which proves \eqref{eq:lower-reg} after enlarging the constant in its statement.
\end{itemize}
\end{proof}

We next obtain matching upper bounds by elementary test functions. In the pole case take
\[
 v(t,y)=\cos\left(\frac{\pi(t-t_0)}{2(R-t_0)}\right)
 \quad(t_0<t<R),\qquad v=1\quad\text{on }K.
\]
In the analytic case take
\[
 v(t,y)=\sin\left(\frac{\pi(t-t_0)}{R-t_0}\right)
 \quad(t_0<t<R),\qquad v=0\quad\text{on }K.
\]
In both cases the traces agree at $t=t_0$, the function vanishes at $t=R$, and hence $v\in H_0^1(X_R)$. In the pole case its core energy vanishes because $W=0$ and $v$ is constant on $K$. In the analytic case both its core gradient energy and its $W$-energy vanish. The possible jump in the first normal derivative at $t=t_0$ is allowed in $H^1$ and produces no additional term in the quadratic form.

Write $c(X)=\frac{\pi^2}{4}$ in the pole case and $c(X)=\pi^2$ in the analytic case. For the corresponding test function, direct integration gives
\[
 \int_{t_0}^R v^2\,dt=\frac{R-t_0}{2},\qquad
 \int_{t_0}^R|v_t|^2\,dt=\frac{c(X)}{2(R-t_0)}.
\]
There is no tangential energy. Since $|v|\le1$ and
$|v_t|\le\frac{\pi}{R-t_0}$, \eqref{eq:weighted-density} yields
\begin{align*}
 \left|\int_{X_R\setminus K}v^2\,d\mu
       -\frac{R-t_0}{2}\int_Mb^2\,dv_h\right|
 &\le C\int_{t_0}^\infty e^{-2t}\,dt\le C,\\
 \left|q(v)-\frac{c(X)}{2(R-t_0)}\int_Mb^2\,dv_h\right|
 &\le\frac{C}{(R-t_0)^2}\int_{t_0}^\infty e^{-2t}\,dt
 \le\frac{C}{(R-t_0)^2}.
\end{align*}
The core mass is either $\mu(K)$ or zero, and hence only contributes $O(1)$. Because $\int_Mb^2\,dv_h>0$, it follows that
\begin{align*}
 \lambda_1(X_R)-\frac{n^2}{4}
 &\le\frac{q(v)}{\displaystyle\int_{X_R}v^2\,d\mu}\\
 &=\frac{\displaystyle\frac{c(X)}{2(R-t_0)}
                     \int_Mb^2\,dv_h+O(R^{-2})}
          {\displaystyle\frac{R-t_0}{2}\int_Mb^2\,dv_h+O(1)}\\
 &=\frac{c(X)}{(R-t_0)^2}+O(R^{-3})
  =\frac{c(X)}{R^2}+O(R^{-3}).
\end{align*}
The lower bounds \eqref{eq:lower-res}--\eqref{eq:lower-reg} give
$\lambda_1(X_R)-\frac{n^2}{4}\ge\frac{c(X)}{R^2}-CR^{-3}$.
Together these estimates prove
\[
 \lambda_1(X_R)=\frac{n^2}{4}+\frac{c(X)}{R^2}+O(R^{-3}).
\]
The two constants are distinct, and the two resolvent cases exhaust all possibilities. Once we pass from $X_R$ to balls, this also proves both reverse implications in \eqref{eq:ball-dichotomy}.

\begin{proof}[Completion of the proof of Theorem~\ref{thm:balls}\textup{(ii)}]
For every fixed $o$, the normal coordinate satisfies
$|d_{g^+}(o,q)-t(q)|\le C_o$ on the collar. Indeed, following a normal ray from the compact section $\{t=t_0\}$ gives the upper bound. For the lower bound, use $|dt|_{g^+}=1$ on the portion of a path after its last crossing of that section. If $o$ itself lies in the collar and a path never crosses the section, the same estimate starts at $o$ and gives $d_{g^+}(o,q)\ge t(q)-t(o)$. Enlarging $C_o$ to control the fixed compact region, we obtain, for large $R$,
\[
 X_{R-C_o}\subset B(o,R)\subset X_{R+C_o}.
\]
For open sets $\Omega_1\subset\Omega_2$, extension by zero embeds $H_0^1(\Omega_1)$ into $H_0^1(\Omega_2)$ and preserves the Rayleigh quotient. Thus domain monotonicity gives
\[
 \lambda_1(X_{R+C_o})\le\lambda_1(B(o,R))
 \le\lambda_1(X_{R-C_o}).
\]
For either sign, the already proved expansion and a fixed-radius shift imply
\begin{align*}
 \lambda_1(X_{R\pm C_o})
 &=\frac{n^2}{4}+\frac{c(X)}{(R\pm C_o)^2}
       +O\bigl((R\pm C_o)^{-3}\bigr)\\
 &=\frac{n^2}{4}+\frac{c(X)}{R^2}+O(R^{-3}).
\end{align*}
Squeezing between these bounds proves \eqref{eq:ball-dichotomy}. In particular, $c(X)$ is independent of $o$, although the error bound may depend on $o$. The argument uses only the variational definition on $H_0^1$. It does not require smoothness of the distance function or of the geodesic-sphere boundary.
\end{proof}

The proof of \cite[Theorem 3.4]{GQ} also gives $S(\frac{n}{2})=-I+2\Pi$, where $\Pi$ projects onto boundary traces of the resolvent residue. Nonnegative Yamabe infinity forces $\Pi=0$. Our variational proof converts the resolvent alternative into the two large-ball coefficients without angular scattering estimates.

\begin{remark}\label{rem:alternative-ball-lower-bound}
The lower bounds for $\lambda_1(X_R)$ also admit a proof using positive functions and Barta's inequality. In the analytic case, take the nonnegative $f\in C_c^\infty(X)$ fixed in Subsection~2.2, choose a sufficiently large constant $C>0$, and set
\[
 r=\frac{\pi}{R+C},\qquad
 \psi=-\operatorname{Im}\left[\mathbf R\left(\frac n2+ir\right)f\right].
\]
By a careful analysis of the resolvent near the threshold, one can prove that $\psi>0$ on $\overline{X_R}$ for a suitable fixed $C$ and all sufficiently large $R$. Since $-\Delta_{g^+}\psi=(\frac{n^2}{4}+r^2)\psi$, Barta's inequality then gives $\lambda_1(X_R)\ge\frac{n^2}{4}+\frac{\pi^2}{(R+C)^2}$.

In the pole case, normalize the positive resonant state $\phi$ by $\max_X\phi=1$. Choose a fixed $c>\frac{\pi}{n}$ and put
\[
 a=\frac{\pi}{2(R+C)}-\frac{c}{(R+C)^2},\qquad
 \psi=\phi\sin\left[a\left(R+C+\frac2n\ln\phi\right)\right].
\]
We use the boundary expansion of $\phi$ and the sharp logarithmic gradient estimate $|d\ln\phi|_{g^+}\le\frac n2$~\cite[Corollary 2.3]{SW}. A careful calculation shows that, for sufficiently large $C$ and $R$, $\psi$ is positive on $\{\phi>e^{-n(R+C)/2}\}\supset\overline{X_R}$, vanishes at its boundary, and satisfies $-\Delta_{g^+}\psi/\psi\ge\frac{n^2}{4}+a^2$ there. Barta's inequality and domain monotonicity then give
\[
 \lambda_1(X_R)\ge\frac{n^2}{4}+a^2
 =\frac{n^2}{4}+\frac{\pi^2}{4R^2}+O(R^{-3}).
\]
Increasing the fixed constant in \eqref{eq:lower-res} absorbs the term
$\frac{c}{(R+C)^2}$ in $a$ and yields the lower bound in
Lemma~\ref{lem:truncated-lower-bounds}. We omit the detailed positivity
and differential-inequality computations.
\end{remark}

\subsection{Negative Yamabe examples}\label{sec:examples}
We first explain how Theorem~\ref{thm:balls} replaces a separate large-ball estimate for convex cocompact hyperbolic quotients. For a torsion-free non-elementary convex cocompact group $\Gamma$ with nonempty domain of discontinuity $\Omega_\Gamma$, let
\[
 \delta_\Gamma=\inf\left\{s>0:\sum_{\gamma\in\Gamma}e^{-s d(o,\gamma o)}<\infty\right\}.
\]
The quotient $X=\Gamma\backslash\mathbb H^{n+1}$ is smoothly conformally compact. The relation between the critical exponent and the bottom of the spectrum is classical; see Sullivan~\cite{Sullivan}.

The following proposition applies Theorem~\ref{thm:balls} to the classical relations between the critical exponent, the Yamabe sign, and the resolvent. Its large-ball expansions follow from these relations and our theorem.

\Needspace{10\baselineskip}
\begin{proposition}\label{prop:critical-exponent-balls}
Assume $n\ge3$. For every fixed $o\in X$, the following hold.
\begin{enumerate}
\item If
\begin{equation}\label{eq:window}
 \frac{n-2}{2}<\delta_\Gamma<\frac{n}{2},
\end{equation}
then $Y(M,[h])<0$, $\lambda_1(X)=\frac{n^2}{4}$, and \eqref{eq:counterexample} holds.
\item If $\delta_\Gamma=\frac{n}{2}$, then $Y(M,[h])<0$, $\lambda_1(X)=\frac{n^2}{4}$, and
\[
 \lambda_1(B(o,R))=\frac{n^2}{4}+\frac{\pi^2}{4R^2}+O(R^{-3}).
\]
\end{enumerate}
\end{proposition}

\begin{proof}
The canonical metric on the conformal boundary has scalar-curvature sign equal to the sign of $n-2-2\delta_\Gamma$; see Nayatani~\cite{Nayatani} and the explicit formula in Izeki--Nayatani~\cite[(2.1)]{IN}. This is negative in both cases, so $Y(M,[h])<0$. Sullivan's spectral formula~\cite[Theorem 2.17]{Sullivan} gives $\lambda_1(X)=\frac{n^2}{4}$ whenever $\delta_\Gamma\le\frac{n}{2}$.

First suppose $\delta_\Gamma<\frac{n}{2}$. Patterson's resolvent theorem, as stated in \cite[Proposition 2.3]{GN}, implies that $\mathbf R(s)$ is analytic for $\Re s>\delta_\Gamma$, in particular at $s=\frac{n}{2}$. Theorem~\ref{thm:balls}\textup{(ii)} gives \eqref{eq:counterexample}.

Now suppose $\delta_\Gamma=\frac{n}{2}$. Let $\mu_\Gamma$ be a nonzero Patterson--Sullivan conformal measure supported on the limit set $\Lambda_\Gamma$. Its Poisson integral
\[
 u_\delta(q)=\int_{\Lambda_\Gamma}\mathcal P(\widetilde q,\xi)^{\delta_\Gamma}\,d\mu_\Gamma(\xi)
\]
descends to a positive function on $X$ satisfying $(-\Delta_{g^+}-\frac{n^2}{4})u_\delta=0$; see \cite[Theorem 2.19]{Sullivan}. Near any point of the domain of discontinuity, the boundary variable stays a positive distance from $\Lambda_\Gamma$. In an upper half-space chart the Poisson kernel is $x/(x^2+|y-\xi|^2)$ up to a smooth positive factor depending on $\xi$. Hence $u_\delta=x^{\frac{n}{2}}F$, where $F\in C^\infty(\overline X)$ and $F|_M>0$. By \cite[Lemma 3.2]{GQ}, the continued resolvent has a pole at $\frac{n}{2}$. Theorem~\ref{thm:balls}\textup{(ii)} gives the stated expansion.
\end{proof}

Lee~\cite[Introduction]{Lee} discussed the interval \eqref{eq:window} and examples obtained by extending lower-dimensional cocompact lattices. For even boundary dimensions, we use the convex cocompact groups with critical exponents approaching the boundary dimension from below constructed in Magee~\cite[Section 6, pp.~182--184]{Magee} (following a construction communicated by McMullen).

\begin{proof}[Proof of Remark~\ref{prop:counterexample}, part~1]
If $n=2k+1$, where $k\ge1$, take a closed hyperbolic manifold $(N^{k+1},h_N)$ and extend its lattice trivially from $\mathbb H^{k+1}$ to $\mathbb H^{2k+2}$. Its critical exponent is $\delta_\Gamma=k$, and
\[
 \frac{n-2}{2}<\delta_\Gamma=k<\frac{n}{2}.
\]
Proposition~\ref{prop:critical-exponent-balls}\textup{(1)} gives $Y(M,[h])<0$ and \eqref{eq:counterexample}.

Furthermore, the quotient has the explicit metric
\begin{equation}\label{eq:explicitcounter}
 g^+=dt^2+\sinh^2t\,g_{\mathbb S^k}+\cosh^2t\,h_N,
 \qquad t\ge0.
\end{equation}
The sphere collapses smoothly at $t=0$. The manifold is diffeomorphic to $\mathbb R^{k+1}\times N$. It has constant sectional curvature $-1$. Near infinity, set $x=2e^{-t}$ and extend it as a positive smooth function over the interior. In this collar, the compactification is
\[
 x^2g^+=dx^2+\left(1-\frac{x^2}{4}\right)^2g_{\mathbb S^k}
 +\left(1+\frac{x^2}{4}\right)^2h_N.
\]
The boundary representative $g_{\mathbb S^k}+h_N$ has scalar curvature
$k(k-1)-k(k+1)=-2k<0$.

If $n=2k$, then $k\ge2$. Apply Magee's construction~\cite[Section 6, pp.~182--184]{Magee} in $\mathbb H^{k+1}$ (his dimension parameter is $k$). It gives torsion-free convex cocompact subgroups $\Gamma_j$ of a cocompact lattice with $\delta_{\Gamma_j}\to k$ from below. Indeed, their convex cores are compact and their limit sets have Hausdorff dimensions approaching $k$; the Patterson--Sullivan dimension formula identifies these dimensions with $\delta_{\Gamma_j}$. Since these groups have nonempty domains of discontinuity, $\delta_{\Gamma_j}<k$. Choose $j$ so large that $k-1<\delta_{\Gamma_j}<k$. Embed $\mathbb H^{k+1}$ totally geodesically in $\mathbb H^{2k+1}$ and extend the action by the identity on the normal directions. The subgroup remains torsion-free and convex cocompact. Its orbit distances are unchanged, so its critical exponent is also unchanged. Thus
\[
 \frac{n-2}{2}=k-1<\delta_\Gamma<k=\frac{n}{2}.
\]
Proposition~\ref{prop:critical-exponent-balls}\textup{(1)} again gives $Y(M,[h])<0$ and \eqref{eq:counterexample} for every fixed center.
\end{proof}

There is also a simple negative-Yamabe example in the resonant branch. Let $(N^3,h_N)$ be a closed hyperbolic manifold and put, on $\mathbb R^2\times N$,
\begin{equation}\label{eq:resonant-example}
 g^+=dt^2+\sinh^2t\,d\theta^2+\cosh^2t\,h_N,\qquad t\ge0.
\end{equation}
The circle collapses smoothly at zero, and the warped-product curvature formulas give sectional curvature $-1$. Near infinity, set $x=2e^{-t}$ and extend it as a positive smooth function over the interior. In this collar,
\[
 x^2g^+=dx^2+\left(1-\frac{x^2}{4}\right)^2d\theta^2
 +\left(1+\frac{x^2}{4}\right)^2h_N.
\]
Here the group is the lattice of $N^3$, extended from $\mathbb H^3$ to $\mathbb H^5$, and $\delta_\Gamma=2=\frac{n}{2}$. Proposition~\ref{prop:critical-exponent-balls}\textup{(2)} therefore gives $Y(M,[h])<0$, $\lambda_1(X)=4$, and
\[
 \lambda_1(B(o,R))=4+\frac{\pi^2}{4R^2}+O(R^{-3}).
\]
The boundary metric $d\theta^2+h_N$ has scalar curvature $-6$. An explicit threshold resonant state is $u(t)=\cosh^{-2}t$: in the collar it takes the form
\[
 u=x^2\left(1+\frac{x^2}{4}\right)^{-2}
\]
and a direct calculation gives
\[
 \Delta_{g^+}u=u''+(\coth t+3\tanh t)u'=-4u.
\]
Thus $Pu=0$ and $u\in x^2C^\infty(\overline X)$. This one-ended example with connected conformal infinity realizes the second branch of the theorem, whereas Remark~\ref{prop:counterexample} realizes the first. The existence of a solution to $Pu=0$ alone does not distinguish these cases. Its boundary behavior is essential.

\subsection{Other coefficients under a Ricci lower bound}\label{sec:ricci-examples}
For a complete $(n+1)$-manifold $(X,g)$ with $\Ric_g\ge-ng$ and $\lambda_1(X)=\frac{n^2}{4}$, Cheng's comparison~\cite{Cheng} and Savo's hyperbolic ball asymptotics~\cite{Savo} give
\begin{equation}\label{eq:general-upper}
 0\le\lambda_1(B(o,R))-\frac{n^2}{4}
 \le\frac{\pi^2}{R^2}+O(R^{-3}).
\end{equation}
Thus $R\bigl(\lambda_1(B(o,R))-\frac{n^2}{4}\bigr)\to0$, without assuming a full expansion. The next coefficient, when it exists, lies in $[0,\pi^2]$. We now realize a continuous interval of values different from the two in Theorem~\ref{thm:balls}, in every dimension at least three.

Fix $n\ge2$ and $0<\kappa<\frac{1}{4}$. On $\mathbb R^{n+1}$ define the rotationally symmetric metric
\begin{equation}\label{eq:ricci-metric}
 g_\kappa=dr^2+f(r)^2g_{\mathbb S^n},\qquad
 f(r)=\sinh r\exp\left[\frac{2\kappa}{n^2}
 \left(\frac{1}{\sqrt{1+r^2}}-1\right)\right].
\end{equation}
The function $f$ is smooth and odd at zero, with $f'(0)=1$, so the metric extends smoothly over the pole $o_0$ and is complete. We prove
\begin{equation}\label{eq:ricci-asymptotic}
 \lambda_1(\mathbb R^{n+1},g_\kappa)=\frac{n^2}{4},\qquad
 \lambda_1(B(o,R))=\frac{n^2}{4}
 +\frac{\left(j_{\sqrt{\frac{1}{4}-\kappa},1}\right)^2}{R^2}+o(R^{-2})
\end{equation}
for every fixed center $o$, where $j_{\nu,1}$ is the first positive zero of the Bessel function $J_\nu$.

First we check curvature. Put $b(r)=\frac{r}{(1+r^2)^{\frac{3}{2}}}$. Then
\[
 \frac{f'}f=\coth r-\frac{2\kappa}{n^2}b>0,\qquad
 \frac{f''}f=1-\frac{2\kappa}{n^2}(2b\coth r+b')
 +\frac{4\kappa^2}{n^4}b^2.
\]
The elementary inequality $2b\coth r+b'\ge b^2$ holds for all $r>0$. Indeed, multiplication by $(1+r^2)^{\frac{5}{2}}$ reduces it to
\[
 2r(1+r^2)\coth r+1-2r^2\ge\frac{r^2}{\sqrt{1+r^2}}.
\]
For $r\le1$ the left side is at least $3$, by $r\coth r\ge1$; for $r\ge1$ it is at least $2r^3-2r^2+2r+1\ge2r+1$, by $\coth r\ge1$. Both bounds prove the inequality. Since $\frac{2\kappa}{n^2}<1$, we have $f''\le f$. Integrating
$(f'^2-f^2)'=2f'(f''-f)\le0$ gives $f'^2\le1+f^2$. The radial and tangential sectional curvatures are therefore
\[
 -\frac{f''}f\ge-1,\qquad \frac{1-f'^2}{f^2}\ge-1,
\]
and in particular $\Ric_{g_\kappa}\ge-ng_\kappa$.

Rotational invariance and simplicity make the first ball eigenfunction at $o_0$ radial. The unitary change $v=f^{\frac{n}{2}}u$ transforms its shifted equation into the Friedrichs problem for $-\partial_r^2+Q(r)$ on $(0,R)$, where
\begin{align}\label{eq:inverse-square}
 Q(r)&=\frac{(f^{\frac{n}{2}})''}{f^{\frac{n}{2}}}-\frac{n^2}{4}\notag\\
 &=\frac{n(n-2)}{4\sinh^2r}-\frac{\kappa}{n}b'
 -\kappa b\coth r+\frac{\kappa^2}{n^2}b^2\notag\\
 &=-\frac{\kappa}{r^2}+O(r^{-3}).
\end{align}
Crucially, the pointwise bound $Q(r)\ge-\frac{\kappa}{r^2}$ holds on the whole interval. To see it, use
$\coth r\le\frac{\sqrt{1+r^2}}{r}$ and $b'\le(1+r^2)^{-\frac{5}{2}}$ to obtain
\[
 b\coth r+\frac{1}{n} b'
 \le\frac{1}{1+r^2}+\frac{1}{n(1+r^2)^{\frac{5}{2}}}
 \le\frac{1}{r^2};
\]
the last inequality follows from $r^2\le n(1+r^2)^{\frac{3}{2}}$. The other terms in \eqref{eq:inverse-square} are nonnegative.

Hardy's inequality gives
\begin{equation}\label{eq:Hardy-model}
 \int_0^R\left(|v'|^2+Qv^2\right)dr
 \ge(1-4\kappa)\int_0^R|v'|^2dr\ge0,
 \qquad v\in H_0^1(0,R).
\end{equation}
Near zero, $Q(r)=\frac{n(n-2)}{4r^2}+O(1)$, so Hardy's inequality also identifies the Friedrichs form domain with $H_0^1(0,R)$. It follows that every centered ball has eigenvalue at least $\frac{n^2}{4}$. Exhaustion and \eqref{eq:general-upper}, with its upper comparison independent of the assumption on $\lambda_1$, prove the first assertion of \eqref{eq:ricci-asymptotic}.

After scaling $r=Rs$, the shifted eigenvalue satisfies
\[
 R^2\left(\lambda_1(B(o_0,R))-\frac{n^2}{4}\right)
 =\inf\limits_{0\ne v\in H_0^1(0,1)}
 \frac{\int_0^1\bigl(|v'|^2+R^2Q(Rs)v^2\bigr)ds}
 {\int_0^1v^2ds}.
\]
The pointwise lower bound for $Q$ bounds this quotient below by the first Friedrichs eigenvalue of $-\partial_s^2-\frac{\kappa}{s^2}$. Conversely, for any fixed test function in $C_c^\infty(0,1)$, \eqref{eq:inverse-square} gives convergence of the quotient to the corresponding one for this limiting operator. Density in $H_0^1$, justified by $\kappa<\frac{1}{4}$ and Hardy's inequality, proves convergence of the infima. The limiting operator is the classical Bessel model; see Watson~\cite{Watson}. Its first eigenvalue is $(j_{\sqrt{\frac{1}{4}-\kappa},1})^2$, and its first eigenfunction is proportional to
$\sqrt{s}\,J_{\sqrt{\frac{1}{4}-\kappa}}(j_{\sqrt{\frac{1}{4}-\kappa},1}s)$.
This proves \eqref{eq:ricci-asymptotic} at $o_0$. Domain monotonicity between balls with radii $R\pm d(o,o_0)$ proves it for every fixed center $o$.

As $\kappa$ varies in $(0,\frac{1}{4})$, the coefficients fill
$((j_{0,1})^2,\pi^2)$, since the first Friedrichs eigenvalue is continuous and strictly decreasing in $\kappa$, with endpoint values $(j_{0,1})^2$ and $(j_{\frac{1}{2},1})^2=\pi^2$. Here we use the continuity and strict increase of $j_{\nu,1}$ with respect to $\nu\ge0$; see \cite[pp.~508--510]{Watson}. In particular they differ from both $\pi^2$ and $\frac{\pi^2}{4}$. These examples are not Einstein: $f''<f$ for $r>0$, so their radial Ricci curvature is strictly greater than $-n$, whereas it tends to $-n$ at infinity. These metrics prove the assertion under a Ricci lower bound in Remark~\ref{prop:counterexample}.

\section{Sharp large-ball asymptotics for the \texorpdfstring{$p$}{p}-Laplacian}\label{sec:p-balls}
We prove Theorem~\ref{thm:p-balls} for a fixed $p\in(1,\infty)$. Write
$\Delta_pu=\operatorname{div}_{g^+}(|du|_{g^+}^{p-2}\nabla u)$.
For the lower bound, we use Sakurai's one-dimensional comparison theorem~\cite[Theorem 1.6]{Sakurai} and calculate the precise asymptotics of its model eigenvalue. For the upper bound, we refine the hyperbolic test function of
\cite[Lemma 4.2]{Jinp}.

\begin{proof}[Proof of Theorem~\ref{thm:p-balls}]
\emph{The one-dimensional model.}
For $L>0$, denote by $\mu_L$ the first eigenvalue of
\begin{equation}\label{eq:p-model}
 \begin{cases}
 \bigl(e^{nt}|\varphi'|^{p-2}\varphi'\bigr)'
       +\mu_L e^{nt}\varphi^{p-1}=0,&0<t<L,\\
 \varphi'(0)=0,\qquad \varphi(L)=0.
 \end{cases}
\end{equation}
This is the unweighted case of the model in \cite[(1.6)]{Sakurai}.

Its first eigenfunction $\varphi$ is positive on $[0,L)$, and
integration of the equation from $0$ to $t$ shows that $\varphi'<0$
on $(0,L)$. Put
\[
 z=\left(-\frac{\varphi'}{\varphi}\right)^{p-1},\qquad
 |\varphi'|^{p-2}\varphi'=-z\varphi^{p-1}.
\]
Substitution into \eqref{eq:p-model} gives
\[
 -\bigl(e^{nt}z\varphi^{p-1}\bigr)'
 =-e^{nt}\varphi^{p-1}
   \left(z'+nz-(p-1)z^{\frac{p}{p-1}}\right)
 =-\mu_Le^{nt}\varphi^{p-1}.
\]
Thus
\[
 z'+nz-(p-1)z^{\frac{p}{p-1}}=\mu_L,\qquad
 z(0)=0,\qquad z(L^-)=+\infty.
\]
Introduce
\[
 \delta=\mu_L-\left(\frac{n}{p}\right)^p,\qquad
 \Phi(z)=\left(\frac{n}{p}\right)^p-nz
               +(p-1)z^{\frac{p}{p-1}},
 \qquad z'=\delta+\Phi(z).
\]
Direct differentiation gives
\[
 \Phi'(z)=-n+pz^{\frac{1}{p-1}},\qquad
 \Phi''(z)=\frac{p}{p-1}z^{\frac{2-p}{p-1}}>0,
 \qquad
 \Phi'''(z)=\frac{p(2-p)}{(p-1)^2}
                z^{\frac{3-2p}{p-1}}\quad(z>0).
\]
Thus $\Phi$ has its unique minimum at
$z_0=\left(\frac{n}{p}\right)^{p-1}$, where it vanishes. The trajectory
must cross this point in finite time, so $\delta>0$. Therefore $z'>0$,
and changing variables from $t$ to $z$ gives
\begin{equation}\label{eq:p-travel-time}
 L=\int_0^\infty\frac{dz}{\delta+\Phi(z)}.
\end{equation}
Since
\eqref{eq:p-travel-time} is uniformly bounded when $\delta$ stays
away from zero, $\delta\to0$ as $L\to\infty$.

For the exact remainder, write $w=z-z_0$. Taylor's formula, including
the cubic term, gives
\begin{align*}
 \Phi(z_0+w)
 &=\frac{p}{2(p-1)}\left(\frac{n}{p}\right)^{2-p}w^2
  +\frac{p(2-p)}{6(p-1)^2}
       \left(\frac{n}{p}\right)^{3-2p}w^3+O(w^4).
\end{align*}
On a fixed neighborhood of this zero, set
\[
 s=\operatorname{sgn}(w)\sqrt{\Phi(z)}.
\]
The sign makes $s$ a smooth coordinate on both sides of the zero.
Taking the square root of the preceding Taylor expansion gives
$$\begin{cases}
  s=\sqrt{\frac{p}{2(p-1)}}
   \left(\frac{n}{p}\right)^{\frac{2-p}{2}}w
   \left[1+\frac{2-p}{6(p-1)}
       \left(\frac{n}{p}\right)^{1-p}w+O(w^2)\right]\\
   \frac{ds}{dz}
 =\sqrt{\frac{p}{2(p-1)}}
   \left(\frac{n}{p}\right)^{\frac{2-p}{2}}
   \left[1+\frac{2-p}{3(p-1)}
       \left(\frac{n}{p}\right)^{1-p}w+O(w^2)\right]\\
 w=\sqrt{\frac{2(p-1)}{p}}
   \left(\frac{n}{p}\right)^{\frac{p-2}{2}}s+O(s^2)\\
 \frac{dz}{ds}=\left(\frac{ds}{dz}\right)^{-1}
  =\sqrt{\frac{2(p-1)}{p}}
   \left(\frac{n}{p}\right)^{\frac{p-2}{2}}
   -\frac{2(2-p)}{3n}s+O(s^2).
\end{cases}$$

Choose a sufficiently small fixed $\varepsilon>0$ so that $s$ has a
smooth local inverse $z=z(s)$ on $[-\varepsilon,\varepsilon]$, and put
$z_-=z(-\varepsilon)$ and $z_+=z(\varepsilon)$. Then
$0<z_-<z_0<z_+$, and changing variables only on $[z_-,z_+]$
splits \eqref{eq:p-travel-time} exactly as
\[
 L=\int_0^{z_-}\frac{dz}{\delta+\Phi(z)}
   +\int_{-\varepsilon}^{\varepsilon}
        \frac{dz/ds}{\delta+s^2}\,ds
   +\int_{z_+}^{\infty}\frac{dz}{\delta+\Phi(z)}.
\]
The two outer integrals are uniformly bounded for $\delta>0$:
\[
 0\le\int_0^{z_-}\frac{dz}{\delta+\Phi(z)}
   +\int_{z_+}^{\infty}\frac{dz}{\delta+\Phi(z)}
 \le\int_0^{z_-}\frac{dz}{\Phi(z)}
   +\int_{z_+}^{\infty}\frac{dz}{\Phi(z)}<\infty.
\]
Indeed, $\Phi$ has no zeros on these intervals and grows like
$z^{\frac{p}{p-1}}$ at infinity; the last exponent exceeds one.
For the middle integral, the preceding expansion of $dz/ds$ gives
\begin{align*}
 \int_{-\varepsilon}^{\varepsilon}
       \frac{dz/ds}{\delta+s^2}\,ds
 &=\sqrt{\frac{2(p-1)}{p}}
   \left(\frac{n}{p}\right)^{\frac{p-2}{2}}
   \int_{-\varepsilon}^{\varepsilon}
       \frac{ds}{\delta+s^2}\\
 &\quad-\frac{2(2-p)}{3n}
   \int_{-\varepsilon}^{\varepsilon}
       \frac{s}{\delta+s^2}\,ds
   +\int_{-\varepsilon}^{\varepsilon}
       \frac{O(s^2)}{\delta+s^2}\,ds.
\end{align*}
The three integrals on the right satisfy
$$\begin{cases}
 \int_{-\varepsilon}^{\varepsilon}
       \frac{ds}{\delta+s^2}
 =\frac{2}{\sqrt{\delta}}
       \arctan\left(\frac{\varepsilon}{\sqrt{\delta}}\right)
   =\frac{\pi}{\sqrt{\delta}}+O(1)\\
 \int_{-\varepsilon}^{\varepsilon}
       \frac{s}{\delta+s^2}\,ds=0\\
 \int_{-\varepsilon}^{\varepsilon}
       \frac{s^2}{\delta+s^2}\,ds\le 2\varepsilon.
\end{cases}$$
Consequently,
\begin{align}\label{eq:p-model-expansion}
 L&=\sqrt{\frac{2(p-1)}{p}}
     \left(\frac{n}{p}\right)^{\frac{p-2}{2}}
     \frac{\pi}{\sqrt{\delta}}+O(1),\notag\\
 \mu_L&=\left(\frac{n}{p}\right)^p
    +\frac{2(p-1)}{p}\left(\frac{n}{p}\right)^{p-2}
       \frac{\pi^2}{L^2}+O(L^{-3}).
\end{align}
Indeed, the first line gives $\delta=\frac{2(p-1)}{p}\left(\frac{n}{p}\right)^{p-2}\pi^2\bigl(L+O(1)\bigr)^{-2}$. We expand the denominator to obtain the second line. The cancellation of the odd term in $dz/ds$
is what keeps the error in the first line bounded.

\emph{The lower bound.}
Let $\Omega$ be a smooth bounded domain with $\Ric\ge-ng$, outward mean curvature $H_{\partial\Omega}\ge n$, and inradius $L$. Sakurai's comparison theorem~\cite[Theorem 1.6]{Sakurai}, with the parameters specified above, gives
\begin{equation}\label{eq:p-domain-lower}
 \lambda_{1,p}(\Omega)\ge\mu_L.
\end{equation}
The same bound holds for $C^2$ boundaries by smooth approximation and continuity of the model eigenvalue in the mean-curvature parameter and interval length. We use this known comparison without repeating its proof.

Choose a boundary representative of nonnegative scalar curvature.
As recalled at the beginning of Section~2, the $C^3$ conformal compactness
gives a $C^2$ geodesic compactification for this boundary representative.
The proof of the mean-curvature estimate in \cite[Lemma 4.1]{Jinp}
requires only this $C^2$ regularity: the scalar curvature of the
geodesic compactification extends continuously to $M$, and its
Riccati inequality holds in the interior.
The mean-curvature estimate of \cite[Lemma 4.1]{Jinp} and the normal
collar construction in \cite[Section 4.1]{Jinp} provide truncations
$E_R$ satisfying
\[
 B(o,R)\subset E_R,\qquad H_{\partial E_R}\ge n,\qquad
 \operatorname{Inrad}(E_R)=R+C_o
\]
for a constant $C_o$ independent of $R$. The distance and inradius identities are given in \cite[(2.16) and Lemma 2.5]{JinRadius}. By domain monotonicity and
\eqref{eq:p-model-expansion}--\eqref{eq:p-domain-lower},
\begin{equation}\label{eq:p-ball-lower}
 \lambda_{1,p}(B(o,R))\ge
 \left(\frac{n}{p}\right)^p
 +\frac{2(p-1)}{p}\left(\frac{n}{p}\right)^{p-2}
    \frac{\pi^2}{R^2}-O(R^{-3}).
\end{equation}

\emph{The upper bound.}
The Cheng-type comparison under $\Ric\ge-ng$ applies to every
$p>1$ by \cite[Theorem 3.2 and Remark 3.3(2)]{Mao}; see also
\cite[Theorem 3]{Takeuchi}:
\[
 \lambda_{1,p}(B(o,R))
 \le\lambda_{1,p}\bigl(B^{\mathbb H^{n+1}}(R)\bigr).
\]
Following the radial Rayleigh-quotient calculation of
\cite[Lemma 4.2]{Jinp}, with an endpoint modification, set
\[
 v(r)=
 \begin{cases}
 r\sin^{\frac{2}{p}}\!\frac{\pi}{R},&0\le r\le1,\\[3pt]
 \sin^{\frac{2}{p}}\!\frac{\pi r}{R},&1\le r\le R-1,\\[3pt]
 (R-r)\sin^{\frac{2}{p}}\!\frac{\pi}{R},&R-1\le r\le R,
 \end{cases}
 \qquad u(r)=e^{-\frac{n}{p} r}v(r).
\]
The linear endpoint pieces give $u\in W_0^{1,p}
\bigl(B^{\mathbb H^{n+1}}(R)\bigr)$ for all fixed $p>1$.
Since $v(R-r)=v(r)$ and $e^{-nr}(\sinh r)^n$ is a constant multiple
of $(1-e^{-2r})^n$, direct substitution gives
\begin{equation}\label{eq:p-test-quotient}
\begin{aligned}
 \lambda_{1,p}(B(o,R))
 &\le\frac{\displaystyle\int_0^R
   (1-e^{-2r})^n
   \left|\frac{n}{p} v-v'\right|^p\,dr}
  {\displaystyle\int_0^R(1-e^{-2r})^n v^p\,dr}\\
 &\le\frac{\displaystyle\int_0^R
          \left|\frac{n}{p} v-v'\right|^p\,dr}
 {\displaystyle\int_0^R(1-e^{-2r})^n v^p\,dr}.
\end{aligned}
\end{equation}
We first calculate the denominator in \eqref{eq:p-test-quotient}.
The two linear endpoint pieces of $v$ give
\begin{align*}
 \int_0^R v^p\,dr
 &=\int_0^R\sin^2\!\left(\frac{\pi r}{R}\right)dr
  +2\int_0^1\left[
      r^p\sin^2\!\left(\frac{\pi}{R}\right)
      -\sin^2\!\left(\frac{\pi r}{R}\right)
     \right]dr\\
 &=\frac{R}{2}+O(R^{-2}).
\end{align*}
For every integer $1\le k\le n$, the profile satisfies
$v^p\le C(1+r)^2R^{-2}$ on $[0,\frac{R}{2}]$ and
$v^p\le1$ on $[\frac{R}{2},R]$. Therefore
\[
 \int_0^R e^{-2kr}v^p\,dr
 \le\frac{C}{R^2}\int_0^{\frac{R}{2}}
       e^{-2kr}(1+r)^2\,dr
     +\int_{\frac{R}{2}}^\infty e^{-2kr}\,dr
 =O(R^{-2}),\qquad 1\le k\le n.
\]
Expanding $(1-e^{-2r})^n$ by the binomial formula now gives
\begin{align}\label{eq:p-denominator}
 \int_0^R(1-e^{-2r})^nv^p\,dr
 &=\int_0^R v^p\,dr
   +\sum_{k=1}^n(-1)^k\binom nk
        \int_0^R e^{-2kr}v^p\,dr\notag\\
 &=\frac{R}{2}+O(R^{-2}).
\end{align}

We next calculate the numerator. The symmetry $v(R-r)=v(r)$ gives
\[
 \int_0^R\left|\frac{n}{p}v-v'\right|^p\,dr
 =\frac{1}{2}\int_0^R\left(
   \left|\frac{n}{p}v-v'\right|^p
  +\left|\frac{n}{p}v+v'\right|^p\right)dr.
\]
Choose a fixed $d\ge1$ large enough that
$|v'|\le\frac{n}{2p}v$ on $[d,R-d]$.
Indeed, on this interval the sine profile satisfies
\[
 \frac{|v'(r)|}{v(r)}
 =\frac{2\pi}{pR}\left|\cot\frac{\pi r}{R}\right|
 \le\frac{2}{p\min\{r,R-r\}}
 \le\frac{2}{pd}.
\]
Thus any fixed $d\ge\max\{1,\frac{4}{n}\}$ has the stated
property for all sufficiently large $R$.
On $[0,1]$, for example,
\[
 \int_0^1v^{p-2}(v')^2\,dr
 =\sin^2\!\left(\frac{\pi}{R}\right)
  \int_0^1r^{p-2}\,dr
 =\frac{\sin^2(\frac{\pi}{R})}{p-1}=O(R^{-2}).
\]
The other terms on $[0,d]\cup[R-d,R]$ satisfy the same bound,
since $d$ is fixed and $p>1$. In particular,
\[
 \int_{[0,d]\cup[R-d,R]}
 \left(v^p+v^{p-2}(v')^2
          +\left|\frac{n}{p}v-v'\right|^p\right)dr
 =O(R^{-2}).
\]
On $[d,R-d]$, the bound on $|v'|$ allows the even Taylor expansion
\begin{align*}
 &\frac{1}{2}\left(
    \left|\frac{n}{p}v-v'\right|^p
   +\left|\frac{n}{p}v+v'\right|^p\right)\\
 &\quad=\left(\frac{n}{p}\right)^pv^p
  +\frac{p(p-1)}{2}\left(\frac{n}{p}\right)^{p-2}
      v^{p-2}(v')^2
  +O\bigl(v^{p-4}(v')^4\bigr).
\end{align*}
Its remainder is integrable, and
\begin{align*}
 \int_d^{R-d}v^{p-4}(v')^4\,dr
 \le C\left(\frac{\pi}{R}\right)^4
   \int_d^{R-d}
    \frac{\cos^4(\frac{\pi r}{R})}
         {\sin^2(\frac{\pi r}{R})}\,dr
 \le\frac{C}{R^2}\int_d^{\frac{R}{2}}\frac{dr}{r^2}
 =O(R^{-2}).
\end{align*}
The quadratic term can be integrated exactly on all three pieces:
\begin{align*}
 \int_0^Rv^{p-2}(v')^2\,dr
 &=\frac{4\pi^2}{p^2R^2}
      \int_1^{R-1}\cos^2\!\left(\frac{\pi r}{R}\right)dr
     +\frac{2\sin^2(\frac{\pi}{R})}{p-1}\\
 &=\frac{2\pi^2}{p^2R}+O(R^{-2}).
\end{align*}
Combining these calculations with
$\int_0^Rv^p\,dr=\frac{R}{2}+O(R^{-2})$ gives
\begin{equation}\label{eq:p-unweighted-energy}
\begin{aligned}
 \int_0^R\left|\frac{n}{p}v-v'\right|^p\,dr
 &=\left(\frac{n}{p}\right)^p\int_0^Rv^p\,dr
  +\frac{p(p-1)}{2}\left(\frac{n}{p}\right)^{p-2}
     \int_0^Rv^{p-2}(v')^2\,dr+O(R^{-2})\\
 &=\left(\frac{n}{p}\right)^p\frac{R}{2}
  +\frac{p-1}{p}\left(\frac{n}{p}\right)^{p-2}
       \frac{\pi^2}{R}+O(R^{-2}).
\end{aligned}
\end{equation}
Finally, \eqref{eq:p-test-quotient}, \eqref{eq:p-denominator}, and
\eqref{eq:p-unweighted-energy} give
\begin{align}\label{eq:p-ball-upper}
 \lambda_{1,p}(B(o,R))
 &\le\frac{\displaystyle
     \left(\frac{n}{p}\right)^p\frac{R}{2}
   +\frac{p-1}{p}\left(\frac{n}{p}\right)^{p-2}
      \frac{\pi^2}{R}+O(R^{-2})}
     {\displaystyle\frac{R}{2}+O(R^{-2})}\notag\\
 &=\left(\frac{n}{p}\right)^p
 +\frac{2(p-1)}{p}\left(\frac{n}{p}\right)^{p-2}
   \frac{\pi^2}{R^2}+O(R^{-3}).
\end{align}
Combining \eqref{eq:p-ball-lower} and \eqref{eq:p-ball-upper}
proves \eqref{eq:p-ball-expansion}.
\end{proof}

The following proposition strengthens \cite[Proposition 1.5]{Jinp} by using the sharp second coefficient and a larger range of $p$. We use the splitting theorem of Sung--Wang~\cite[Theorem 1.4]{SW} and refine the estimates for the two exceptional models considered in \cite[Section 5]{Jinp}.
\begin{proposition}\label{prop:p-one-end}
Let $(X^{n+1},g)$ be a connected complete noncompact manifold, where $n\ge2$, with $\Ric_g\ge-ng$. Suppose that, for some $o\in X$ and
$ 1<p\le\frac{n^2}{2(n-1)},$ there is a constant $C>0$ such that, for all sufficiently large $R$,
\begin{equation}\label{eq:p-one-end}
 \lambda_{1,p}(B(o,R))\ge
 \left(\frac np\right)^p
 +\frac{2(p-1)}p\left(\frac np\right)^{p-2}
   \frac{\pi^2}{R^2}-\frac{C}{R^3}.
\end{equation}
Then $X$ has exactly one end.
\end{proposition}
\begin{proof}
Exhaustion and the upper bound under $\Ric_g\ge-ng$ proved above imply
$\lambda_{1,p}(X)=\left(\frac np\right)^p$. By the spectral rigidity theorem of Sung--Wang~\cite[Theorem 1.4]{SW}, if $X$ has more than one end, then it is isometric to one of
\[
 (\mathbb R\times N,dt^2+e^{2t}g_N)\quad(n\ge2),
 \qquad
 (\mathbb R\times N,dt^2+\cosh^2t\,g_N)\quad(n=2),
\]
where $N$ is compact. Following \cite[Examples 5.2--5.3]{Jinp}, we estimate the eigenvalues of
$\Omega_L=(-L,L)\times N$. Write $o=(t_o,y_o)$. For either model, a path through the slice $\{t=0\}$ gives
\[
 d_g(o,(t,y))\le |t_o|+\operatorname{diam}(N,g_N)+|t|.
\]
Thus $\Omega_L\subset B(o,L+C_o)$ for the fixed constant
$C_o=|t_o|+\operatorname{diam}(N,g_N)$.

For the exponential model, use $u=e^{-\frac np t}v(t)$, where $v$ is the endpoint-modified amplitude $\cos^{\frac2p}(\frac{\pi t}{2L})$ from the preceding proof. Translation of \eqref{eq:p-unweighted-energy}, with interval length $2L$, gives
\[
 \lambda_{1,p}(\Omega_L)\le
 \frac{\displaystyle\int_{-L}^L\left|\frac np v-v'\right|^p\,dt}
 {\displaystyle\int_{-L}^Lv^p\,dt}
 =\left(\frac np\right)^p
 +\frac{p-1}{2p}\left(\frac np\right)^{p-2}
   \frac{\pi^2}{L^2}+O(L^{-3}).
\]
Taking $L=R-C_o$ contradicts \eqref{eq:p-one-end}, since the second coefficient is one quarter of the assumed coefficient.

For the $\cosh$ model, the allowed range is $1<p\le2$. Put
\[
 w(t)=(\cosh t)^{-1}\cos\frac{\pi t}{2L},\qquad |t|<L.
\]
Since $-\Delta_g(\cosh t)^{-1}=(\cosh t)^{-1}$, the ground-state identity gives
\[
 \int_{\Omega_L}w^2\,dv_g=L\operatorname{Vol}(N,g_N),
 \qquad
 \frac{\displaystyle\int_{\Omega_L}|dw|_g^2\,dv_g}
 {\displaystyle\int_{\Omega_L}w^2\,dv_g}
 =1+\frac{\pi^2}{4L^2}.
\]
For $p=2$, this again contradicts \eqref{eq:p-one-end} after taking $L=R-C_o$.
For $1<p<2$, use $w^{\frac2p}$ as a test function and the inequality
\[
 s^p\le\frac p2s^2+1-\frac p2,\qquad s\ge0,
\]
which is strict unless $s=1$. On a fixed neighborhood of $t=0$, the ratio $|w'|/w$ stays uniformly below $1$ for large $L$. The resulting strict deficit has an integral bounded below by a positive constant independent of $L$. Hence, for some $c>0$,
\[
 \lambda_{1,p}(X)\le
 \frac{\displaystyle\int_{\Omega_L}|d(w^{\frac2p})|_g^p\,dv_g}
 {\displaystyle\int_{\Omega_L}w^2\,dv_g}
 \le\left(\frac2p\right)^p
 \left(1+\frac{p\pi^2}{8L^2}-\frac cL\right)
 <\left(\frac2p\right)^p
\]
for large $L$, contradicting the spectral bottom. Both exceptional models are excluded.
\end{proof}

\section{Proof of the boundary-layer characterization}\label{sec:proof}
Fix the geodesic defining function $x$ in Theorem~\ref{thm:main} and set $r=-\ln x$. Its normal collar coordinates give $g^+=x^{-2}(dx^2+h_x)$, with $h_x$ jointly $C^3$ up to the boundary. Throughout this section the layers are exactly $\{R<r(\cdot)<R+\ell\}$.

We first establish the conformal sign relation and the collar estimates needed in the proof.

\subsection{The conformal sign and differentiated collar jets}
Write $\mu_1=\mu_1(L_h)$ and
\[
 \E_h[u]=\int_M\left(|du|_h^2+\frac{n-2}{4(n-1)}R_hu^2\right)\dv_h.
\]
In the notation above, the definition \eqref{eq:Yamabe-definition} becomes
\begin{equation}\label{eq:Yamabe}
 Y(M,[h])=\frac{4(n-1)}{n-2}
 \inf\limits_{0\ne u\in H^1(M)}
 \frac{\E_h[u]}{\left(\int_M|u|^{\frac{2n}{n-2}}\dv_h\right)^{\frac{n-2}{n}}}.
\end{equation}
We use the standard conformal covariance described in Lee--Parker~\cite{LP}.

\begin{lemma}\label{lem:sign}
For every smooth closed $n$-manifold, $n\geq3$,
$\sgn\mu_1(L_h)=\sgn Y(M,[h])$.
\end{lemma}
\begin{proof}
The first eigenvalue of the conformal Laplacian has the variational
characterization
\[
 \mu_1(L_h)=\inf_{0\ne u\in H^1(M)}
 \frac{\E_h[u]}{\displaystyle\int_Mu^2\,dv_h}.
\]
A negative energy gives a negative quotient in both variational problems, and a negative infimum yields a test function with negative energy. If $\mu_1=0$, the energy is nonnegative and a lowest eigenfunction has zero energy, so $Y=0$. If $\mu_1>0$, then $\|u\|_2^2\leq\mu_1^{-1}\E_h[u]$. Since the potential in $L_h$ is bounded, this implies $\|u\|_{H^1}^2\leq C\E_h[u]$. The compact-manifold Sobolev inequality gives $\|u\|_{\frac{2n}{n-2}}^2\leq C\E_h[u]$, and hence $Y>0$. These cases also prove the converse sign implications. Connectedness is unnecessary.
\end{proof}

\begin{lemma}\label{lem:jets}
Let $P_h$ be the Schouten tensor and let $v$ be the relative volume density:
\[
 P_h=\frac{1}{n-2}\left(\Ric_h-\frac{R_h}{2(n-1)}h\right),
 \qquad v(x,y)=\left(\frac{\det h_x(y)}{\det h(y)}\right)^{\frac{1}{2}}.
\]
Then
\begin{equation}\label{eq:jets}
 \left\{
 \begin{aligned}
 &h_x=h-x^2P_h+O(x^3),\\
 &\ln v=-\frac{R_h}{4(n-1)}x^2+O(x^3),\\
 &\partial_r\ln v=\frac{R_h}{2(n-1)}x^2+O(x^3),\\
 &\partial_r^2\ln v=-\frac{R_h}{n-1}x^2+O(x^3),
 \qquad |d_y\ln v|_h=O(x^2).
 \end{aligned}
 \right.
\end{equation}
The errors are uniform. The remainder in the first line of \eqref{eq:jets} remains $O(x^3)$ after up to two applications of $x\partial_x$, and $d_y(h_x-h)=O(x^2)$ in a fixed finite atlas.
\end{lemma}
\begin{proof}
The first line of \eqref{eq:jets} is the second-order
Fefferman--Graham expansion of an Einstein metric in geodesic normal
form; see Graham~\cite{Graham}. At the boundary, the Einstein equation gives
\[
 \left.\partial_xh_x\right|_{x=0}=0,\qquad
 \left.\partial_x^2h_x\right|_{x=0}=-2P_h.
\]
Thus $h_x=h-x^2P_h+R_3(x)$, where the $C^3$ hypothesis and Taylor's theorem
give
\[
 R_3=O(x^3),\qquad \partial_xR_3=O(x^2),\qquad
 \partial_x^2R_3=O(x),\qquad d_y(h_x-h)=O(x^2).
\]
Applying the smooth matrix map
$A\mapsto\frac{1}{2}\ln\det A$ and using
$\tr_hP_h=\frac{R_h}{2(n-1)}$ yields
\[
 \ln v=\frac{1}{2}\ln\det(h^{-1}h_x)
 =-\frac{R_h}{4(n-1)}x^2+O(x^3).
\]
Finally, $\partial_r=-x\partial_x$. Applying this identity once and
twice to the preceding expansion, with the displayed derivative bounds
for $R_3$, proves the remaining lines of \eqref{eq:jets}. The tangential
estimate follows by differentiating the determinant formula in a fixed
finite atlas. We will also use
\[
 D_y^2\ln v=O(x).
\]
Indeed, $\ln v(0,y)=0$, and the mixed derivatives of total order three
are bounded by the $C^3$ hypothesis, so
\[
 D_y^2\ln v(x,y)
 =\int_0^x\partial_\xi D_y^2\ln v(\xi,y)\,d\xi=O(x).
\]
Here $D_y^2$ denotes coordinate derivatives in the fixed finite atlas.
\end{proof}

\subsection{Proof of the layer expansion}
Set
\[
 s(\cdot)=r(\cdot)-R,\qquad \varepsilon=e^{-R}.
\]
Here $r=-\ln x$ is the geodesic coordinate in
Theorem~\ref{thm:main}, and $s$ is its translated layer coordinate.
Recall that, with $x=e^{-r}$,
\[
 g^+=dr^2+e^{2r}h_x,
 \qquad
 v(x,y)=\left(\frac{\det h_x(y)}{\det h(y)}\right)^{\frac12},
\]
and set
\begin{equation}\label{eq:Jdensity}
 J(r,y)=e^{nr}v(e^{-r},y),
 \qquad dv_{g^+}=J(r,y)\,dr\,dv_h.
\end{equation}
Thus $v$ is the relative volume density introduced in
Lemma~\ref{lem:jets}, and $dv_{h_x}=v(x,y)\,dv_h$.

\begin{proof}[Proof of Theorem~\ref{thm:main}]
We first assume that $M$ is connected. Let $\psi_1>0$ be an
$L^2(M,dv_h)$-normalized first eigenfunction of $L_h$:
\begin{equation}\label{eq:boundary-ground-state}
 L_h\psi_1=\mu_1\psi_1.
\end{equation}
Let $\varphi_\varepsilon$ be the positive normalized first Dirichlet eigenfunction
of
\[
 -\partial_s^2+\frac{n^2}{4}+\varepsilon^2\mu_1e^{-2s}
 \quad\text{on }(0,\ell),
\]
and denote its eigenvalue by $\nu_\varepsilon$. Thus
\begin{equation}\label{eq:radial-ground-state}
 \left\{
 \begin{aligned}
 &-\varphi_\varepsilon''+\left(\frac{n^2}{4}
 +\varepsilon^2\mu_1e^{-2s}\right)\varphi_\varepsilon
 =\nu_\varepsilon\varphi_\varepsilon,\\
 &\varphi_\varepsilon(0)=\varphi_\varepsilon(\ell)=0,
 \qquad \varphi_\varepsilon>0\quad\text{on }(0,\ell),\\
 &\int_0^\ell\varphi_\varepsilon^2\dd s=1.
 \end{aligned}
 \right.
\end{equation}

\begin{itemize}
\item[\textup{Step 1.}]
We first establish the expansion of $\nu_\varepsilon$. Put
\[
 H_0=-\partial_s^2+\frac{n^2}{4},\qquad
 \lambda_0=\frac{n^2}{4}+\frac{\pi^2}{\ell^2},
\]
and let
\[
 \phi(s)=\sqrt{\frac{2}{\ell}}\sin\frac{\pi s}{\ell}
\]
be the normalized first eigenfunction of $H_0$. The next eigenvalue is
$\frac{n^2}{4}+\frac{4\pi^2}{\ell^2}$. The variational
characterizations are
\begin{align*}
 \lambda_0
 &=\inf_{\substack{u\in H_0^1(0,\ell)\\ \|u\|_{L^2}=1}}
 \int_0^\ell\left(|u'|^2+\frac{n^2}{4}u^2\right)\dd s,\\
 \nu_\varepsilon
 &=\inf_{\substack{u\in H_0^1(0,\ell)\\ \|u\|_{L^2}=1}}
 \left\{
 \int_0^\ell\left(|u'|^2+\frac{n^2}{4}u^2\right)\dd s
 +\varepsilon^2\mu_1\int_0^\ell e^{-2s}u^2\dd s
 \right\}.
\end{align*}
For every such $u$,
\[
 -\varepsilon^2|\mu_1|
 \leq\varepsilon^2\mu_1\int_0^\ell e^{-2s}u^2\dd s
 \leq\varepsilon^2|\mu_1|.
\]
Subtracting the two variational infima gives
\begin{equation}\label{eq:epsilon-displacement}
 \lambda_0-\varepsilon^2|\mu_1|
 \leq\nu_\varepsilon
 \leq\lambda_0+\varepsilon^2|\mu_1|,
 \qquad
 |\nu_\varepsilon-\lambda_0|\leq\varepsilon^2|\mu_1|.
\end{equation}
In the remainder of Step~1,
\[
 \langle f,g\rangle
 :=\langle f,g\rangle_{L^2(0,\ell)}
 =\int_0^\ell f(s)g(s)\dd s.
\]
Since $\varphi_\varepsilon$ and $\phi$ are positive on $(0,\ell)$,
$c_\varepsilon=\langle\varphi_\varepsilon,\phi\rangle>0$. Write
\[
 \varphi_\varepsilon=c_\varepsilon\phi+z_\varepsilon,
 \qquad z_\varepsilon\perp\phi.
\]
If $Q$ denotes orthogonal projection onto $\phi^\perp$, projection of
\eqref{eq:radial-ground-state} gives
\[
 (H_0-\nu_\varepsilon)z_\varepsilon
 =-\varepsilon^2\mu_1Q(e^{-2s}\varphi_\varepsilon).
\]
The spectrum of $H_0$ on $\phi^\perp$ begins at
$\lambda_0+\frac{3\pi^2}{\ell^2}$. If
$\varepsilon^2|\mu_1|<\frac{3\pi^2}{2\ell^2}$, then
\begin{align*}
 \operatorname{dist}\left(\nu_\varepsilon,
 \sigma(H_0|_{\phi^\perp})\right)
 &\geq\lambda_0+\frac{3\pi^2}{\ell^2}-\nu_\varepsilon\\
 &\geq\frac{3\pi^2}{\ell^2}-\varepsilon^2|\mu_1|
 >\frac{3\pi^2}{2\ell^2}.
\end{align*}
Consequently,
\[
 \left\|(H_0-\nu_\varepsilon)^{-1}
 \right\|_{\phi^\perp\to\phi^\perp}
 =\frac{1}{\operatorname{dist}\left(
 \nu_\varepsilon,\sigma(H_0|_{\phi^\perp})\right)}
 \leq\frac{2\ell^2}{3\pi^2}.
\]
Since $\|Q\|=1$, $0<e^{-2s}\leq1$, and
$\|\varphi_\varepsilon\|_{L^2}=1$, it follows that
\begin{align*}
 \|z_\varepsilon\|_{L^2}
 &\leq\varepsilon^2|\mu_1|
 \left\|(H_0-\nu_\varepsilon)^{-1}\right\|
 \left\|Q(e^{-2s}\varphi_\varepsilon)\right\|_{L^2}\\
 &\leq\frac{2\ell^2}{3\pi^2}\varepsilon^2|\mu_1|.
\end{align*}
The normalization and orthogonality give
\[
 1=\|\varphi_\varepsilon\|_{L^2}^2
 =c_\varepsilon^2+\|z_\varepsilon\|_{L^2}^2,
 \qquad
 c_\varepsilon=(1-\|z_\varepsilon\|_{L^2}^2)^{\frac12}
 =1+O(\varepsilon^4).
\]
Taking the inner product of the eigenvalue equation with $\phi$ yields
\begin{align*}
 \nu_\varepsilon c_\varepsilon
 &=\langle H_0\varphi_\varepsilon,\phi\rangle
 +\varepsilon^2\mu_1
 \langle e^{-2s}\varphi_\varepsilon,\phi\rangle\\
 &=\langle\varphi_\varepsilon,H_0\phi\rangle
 +\varepsilon^2\mu_1
 \langle e^{-2s}(c_\varepsilon\phi+z_\varepsilon),\phi\rangle\\
 &=\lambda_0c_\varepsilon
 +\varepsilon^2\mu_1c_\varepsilon
 \langle e^{-2s}\phi,\phi\rangle
 +\varepsilon^2\mu_1
 \langle e^{-2s}z_\varepsilon,\phi\rangle.
\end{align*}
Moreover,
\[
 \left|\langle e^{-2s}z_\varepsilon,\phi\rangle\right|
 \leq\|z_\varepsilon\|_{L^2}=O(\varepsilon^2),
 \qquad c_\varepsilon^{-1}=1+O(\varepsilon^4).
\]
Consequently,
\begin{align}
 \nu_\varepsilon
 &=\lambda_0+\varepsilon^2\mu_1
 \frac{2}{\ell}\int_0^\ell e^{-2s}
 \sin^2\frac{\pi s}{\ell}\dd s+O(\varepsilon^4)\notag\\
 &=\frac{n^2}{4}+\frac{\pi^2}{\ell^2}
 +\varepsilon^2\radialaverage\mu_1+O(\varepsilon^4).
 \label{eq:nu-epsilon}
\end{align}
All constants are uniform when $\ell$ ranges over a compact subset of
$(0,\infty)$.

\item[\textup{Step 2.}]
Define the separated function
\begin{equation}\label{eq:Barta-test}
 F_R(r,y)=\varphi_\varepsilon(r-R)\psi_1(y).
\end{equation}
Recall that $g^+=dr^2+e^{2r}h_x$, $v=(\det h_x/\det h)^{\frac12}$,
and $J=e^{nr}v$, where $x=e^{-r}$. In the coordinates $(r,y)$, its
Laplacian has the divergence form
\begin{equation}\label{eq:collar-laplacian}
 -\Delta_{g^+}F_R
 =-J^{-1}\partial_r(J\partial_rF_R)
 -x^2v^{-1}\operatorname{div}_h
 \left(vh_x^{-1}d_yF_R\right),
 \qquad x=e^{-r}.
\end{equation}
Set
\[
 W(r,y):=\frac12\Delta_{g^+}\ln J
 -\frac14|d\ln J|_{g^+}^2.
\]
The product rule and \eqref{eq:collar-laplacian} give
\begin{align}
 J^{\frac12}(-\Delta_{g^+})(J^{-\frac12}F_R)
 &=-\Delta_{g^+}F_R+\langle d\ln J,dF_R\rangle_{g^+}
 +W(r,y)F_R\notag\\
 &=-\partial_r^2F_R-(\partial_r\ln J)\partial_rF_R
 -x^2v^{-1}\operatorname{div}_h(vh_x^{-1}d_yF_R)\notag\\
 &\quad+(\partial_r\ln J)\partial_rF_R
 +x^2\langle d_y\ln v,d_yF_R\rangle_{h_x}
 +W(r,y)F_R\notag\\
 &=-\partial_r^2F_R
 -x^2\operatorname{div}_h(h_x^{-1}d_yF_R)
 +W(r,y)F_R.
 \label{eq:conjugated-operator}
\end{align}
Here the last equality uses, with $Y=h_x^{-1}d_yF_R$,
\[
 v^{-1}\operatorname{div}_h(vY)
 =\operatorname{div}_hY+Y(\ln v).
\]

We first calculate the differential part in
\eqref{eq:conjugated-operator}. The separation of variables gives
\begin{align}
 &-\partial_r^2F_R
 -x^2\operatorname{div}_h(h_x^{-1}d_yF_R)\notag\\
 &=-\varphi_\varepsilon''\psi_1
 -x^2\varphi_\varepsilon
 \operatorname{div}_h(h_x^{-1}d_y\psi_1)\notag\\
 &=-\varphi_\varepsilon''\psi_1
 -x^2\varphi_\varepsilon\Delta_h\psi_1
 +O(x^4)\varphi_\varepsilon\psi_1.
 \label{eq:tangential-pointwise}
\end{align}
Indeed, the collar estimates in \eqref{eq:jets} imply
\begin{align*}
 &\operatorname{div}_h(h_x^{-1}d_y\psi_1)-\Delta_h\psi_1=\operatorname{div}_h
 \bigl((h_x^{-1}-h^{-1})d_y\psi_1\bigr),\\
 &\left|
 \operatorname{div}_h
 \bigl((h_x^{-1}-h^{-1})d_y\psi_1\bigr)
 \right|\leq Cx^2
 \left(|\nabla_h^2\psi_1|_h+|d_y\psi_1|_h\right)
 \leq Cx^2\psi_1.
\end{align*}
Elliptic regularity~\cite[Chapter 6]{GT} gives $\psi_1\in C^{2,\beta}(M)$ for every
$0<\beta<1$, since $h\in C^3$. The final inequality follows from this
regularity and $\min_M\psi_1>0$.

We next calculate the potential term. Since
$\partial_r\ln J=n+\partial_r\ln v$ and
$d_y\ln J=d_y\ln v$,
\begin{align}
 W(r,y)
 &=\frac12\partial_r^2\ln v
 +\frac14\left(n+\partial_r\ln v\right)^2\notag\\
 &\quad+\frac{x^2}{2}v^{-1}\operatorname{div}_h
 \left(vh_x^{-1}d_y\ln v\right)
 -\frac{x^2}{4}|d_y\ln v|_{h_x}^2\notag\\
 &=\frac{n^2}{4}
 +\frac{n-2}{4(n-1)}R_hx^2+O(x^3).
 \label{eq:radial-potential-expansion}
\end{align}
The radial terms are given by \eqref{eq:jets}. For the tangential terms,
Lemma~\ref{lem:jets} gives $d_y\ln v=O(x^2)$ and $D_y^2\ln v=O(x)$,
so
\[
 x^2v^{-1}\operatorname{div}_h(vh_x^{-1}d_y\ln v)=O(x^3),
 \qquad x^2|d_y\ln v|_{h_x}^2=O(x^6).
\]
This proves the last equality under the stated $C^3$ regularity.

Combining \eqref{eq:boundary-ground-state},
\eqref{eq:conjugated-operator}--\eqref{eq:radial-potential-expansion},
and $x^2=\varepsilon^2e^{-2s}$ gives
\begin{align}
 &J^{\frac12}(-\Delta_{g^+})(J^{-\frac12}F_R)\notag\\
 &=\left(-\varphi_\varepsilon''
 +\frac{n^2}{4}\varphi_\varepsilon\right)\psi_1
 +x^2\varphi_\varepsilon L_h\psi_1
 +O(x^3)\varphi_\varepsilon\psi_1\notag\\
 &=\left[-\varphi_\varepsilon''
 +\left(\frac{n^2}{4}
 +\varepsilon^2\mu_1e^{-2s}\right)\varphi_\varepsilon\right]\psi_1
 +O(x^3)F_R\notag\\
 &=\nu_\varepsilon F_R+O(x^3)F_R.
 \label{eq:Barta-residual}
\end{align}

Now define the actual test function
\[
 u_R:=J^{-\frac12}F_R.
\]
It belongs to $H_0^1(A_{R,\ell})$, is positive in $A_{R,\ell}$, and
has zero trace on both end faces. Since
$0<x=\varepsilon e^{-s}\leq\varepsilon$ throughout the layer,
\eqref{eq:Barta-residual} yields
\begin{equation}\label{eq:Barta-ratio}
 \left|\frac{-\Delta_{g^+}u_R}{u_R}-\nu_\varepsilon\right|
 \leq C\varepsilon^3
 \qquad\text{in }A_{R,\ell}.
\end{equation}
The factor $F_R=\varphi_\varepsilon\psi_1$ in the residual shows that
the quotient remains uniformly bounded up to the end faces from the
interior.

Using $u_R$ in the Rayleigh quotient and integrating by parts gives
\begin{align}
 \lambda_1(A_{R,\ell})
 &\leq
 \frac{\displaystyle\int_{A_{R,\ell}}|du_R|_{g^+}^2\,dv_{g^+}}
 {\displaystyle\int_{A_{R,\ell}}u_R^2\,dv_{g^+}}\notag\\
 &=\frac{\displaystyle\int_{A_{R,\ell}}
 u_R(-\Delta_{g^+}u_R)\,dv_{g^+}}
 {\displaystyle\int_{A_{R,\ell}}u_R^2\,dv_{g^+}}\notag\\
 &\leq\nu_\varepsilon+C\varepsilon^3.
 \label{eq:Barta-upper}
\end{align}
Barta's inequality and \eqref{eq:Barta-ratio} give the reverse bound
\begin{equation}\label{eq:Barta-lower}
 \lambda_1(A_{R,\ell})
 \geq\inf_{A_{R,\ell}}
 \frac{-\Delta_{g^+}u_R}{u_R}
 \geq\nu_\varepsilon-C\varepsilon^3.
\end{equation}
Consequently,
\begin{equation}\label{eq:Barta-comparison}
 \lambda_1(A_{R,\ell})=\nu_\varepsilon+O(\varepsilon^3).
\end{equation}
Combining \eqref{eq:nu-epsilon} and \eqref{eq:Barta-comparison}, and
using $\varepsilon=e^{-R}$, proves \eqref{eq:main} with the coefficient
\eqref{eq:main-coefficient}.
\end{itemize}

If $M$ is disconnected, apply the argument on each of its finitely many
components. The first eigenvalue of the layer and $\mu_1(L_h)$ are the
minima of the corresponding componentwise first eigenvalues. The
monotonicity of the first one-dimensional eigenvalue in $\mu_1$ and the
uniformity over finitely many components give the same conclusion.
Uniformity for $\ell$ in a compact subset of $(0,\infty)$ follows from
Step~1 and the uniform collar estimates.

Finally, since $\frac{\pi^2(1-e^{-2\ell})}{2\ell(\pi^2+\ell^2)}>0$,
Lemma~\ref{lem:sign} and \eqref{eq:main-coefficient} give
\[
 C_\ell\ge0\quad\Longleftrightarrow\quad\mu_1(L_h)\ge0
 \quad\Longleftrightarrow\quad Y(M,[h])\ge0,
\]
which is \eqref{eq:signintro}.
\end{proof}

\begin{remark}\label{rem:distance-layers}
For an ordinary distance $d=d_{g^+}(K,\cdot)$ from a nonempty compact
set $K$, consider the layers $A^d_{R,\ell}=\{R<d<R+\ell\}$. We do not
establish here whether, without further assumptions, they admit an expansion
\[
 \lambda_1(A^d_{R,\ell})
 =\frac{n^2}{4}+\frac{\pi^2}{\ell^2}
 +C'_\ell e^{-2R}+o(e^{-2R})
\]
with $C'_\ell\ge0$ equivalent to $Y(M,[h])\ge0$.
The coefficient $C'_\ell$, if it exists, need not equal $C_\ell$ for the
fixed representative $h$. For example, if $d=r+b$ near infinity for a
constant $b$, a shift of the layer parameter gives $C'_\ell=e^{2b}C_\ell$.
In hyperbolic space, taking $t=d_{g^+}(o,\cdot)$ and $x=2e^{-t}$ near
infinity gives $r=t-\ln2$ and hence $C'_\ell=4C_\ell$ for $d=t$.

If $d-r=o(x^2)$ uniformly near $M$, however, then
$|d-r|\leq\eta_R=o(e^{-2R})$ on the relevant layers. The inclusions
\[
 \{R+\eta_R<r<R+\ell-\eta_R\}
 \subset A^d_{R,\ell}
 \subset\{R-\eta_R<r<R+\ell+\eta_R\}
\]
and domain monotonicity, together with the uniform remainder for nearby
widths established in the proof of Theorem~\ref{thm:main}, give
\[
 \lambda_1(A^d_{R,\ell})
 =\frac{n^2}{4}+\frac{\pi^2}{\ell^2}
 +C_\ell e^{-2R}+o(e^{-2R}).
\]
Thus $C'_\ell=C_\ell$ under this additional assumption.
For point distances, Li--Qing--Shi~\cite[Lemmas 4.1--4.2, especially (4.5)]{LQS}
establish compactified Lipschitz control and a radial alignment estimate.
Writing $d_o=d_{g^+}(o,\cdot)$, their estimate and $d_o-r=O(1)$ give
\[
 \partial_r(d_o-r)
 =\langle\nabla d_o,\nabla r\rangle_{g^+}-1
 =O(e^{-2r})\quad\text{almost everywhere}.
\]
Integration along collar rays, together with the Lipschitz control, yields
\[
 d_o-r=b_o(y)+O(x^2),
 \qquad b_o\in W^{1,\infty}(M).
\]
For a compact $K$, taking the infimum over $o\in K$ gives the analogous
boundary function $b_K(y)=\inf_{o\in K}b_o(y)$, with uniform constants
on $K$. These estimates do not imply $d-r=o(x^2)$: the boundary
function need not vanish, and the $O(x^2)$ term can affect the
coefficient.
\end{remark}

\medskip
\noindent\textbf{Declaration of AI use.}
AI tools (ChatGPT-6 Astra) were used primarily to assist with the calculations for the counterexamples in Sections~\ref{sec:examples} and~\ref{sec:ricci-examples}, to supplement some details in the proofs of the main theorems, and to improve the typesetting and language. The author has checked the AI-assisted content and takes responsibility for it and for the article as a whole.

The two test functions in Remark~\ref{rem:alternative-ball-lower-bound} were constructed by AI and checked by the author. They give an alternative proof of Lemma~\ref{lem:truncated-lower-bounds} using Barta's inequality. The full calculations would add substantial detail to bounds already proved in Section~2.3. We therefore give only a brief account in the remark and omit the detailed proof.

\medskip
\noindent\textbf{Conflict of interest.}
The author declares that he has no conflict of interest.

\medskip
\noindent\textbf{Data availability.}
Data sharing is not applicable to this article as no datasets
were generated or analysed during the current study.

\bigskip
\noindent Xiaoshang Jin\\
School of Mathematics and Statistics, Huazhong University of Science and Technology,\\
Wuhan, Hubei 430074, China.\\
Email: \href{mailto:jinxs@hust.edu.cn}{jinxs@hust.edu.cn}\par
\end{document}